\documentclass[a4paper,11pt]{article}
\usepackage{amsthm,amsmath,amssymb,amsfonts,bbm,graphics,a4wide,rotating,graphicx}
\usepackage{color}
\usepackage{natbib}

\numberwithin{equation}{section}
\theoremstyle{plain}

\newtheorem{lemma}{Lemma}[section]

\newcommand{\norm}[1]{\ensuremath{\vert\!\vert #1 \vert\!\vert}}
\newcommand{\te}{\theta}
\newcommand{\ep}{\epsilon}

\newcommand{\esp}{\mathbb E\,}

\newcommand{\rn}{\sqrt{n}}

\newcommand{\R}{\mathbb{R}}
\newcommand{\N}{\mathbb{N}}
\newcommand{\Ne}{\mathbb{N}}
\newcommand{\E}{\mathbb{E}}
\newcommand{\Pb}{\mathbb P\, }
\newcommand{\Pp}{\mathbb P }
\renewcommand{\P}{\mathbb P }
\newcommand{\G}{G}
\renewcommand{\L}{\mathbb{L}}
\newcommand{\La}{\Lambda}
\newcommand{\la}{\lambda}

\newcommand{\thol}{\theta_{0\lambda}}
\newcommand{\thoL}{\theta_{0K_n}}
\newcommand{\thoLL}{\theta_{0\bar{K_n}}}
\newcommand{\foL}{f_{0K_n}}
\newcommand{\foLL}{f_{0\bar{K_n}}}
\newcommand{\thl}{\theta_{\lambda}}
\renewcommand{\th}{\theta}
\newcommand{\ga}{\gamma}
\newcommand{\be}{\beta}
\def\1{1\!{\rm l}}
\renewenvironment{proof}{\noindent{\bf Proof.}}{\hfill
  $\blacksquare$\par\noindent}
\newtheorem{thm}{Theorem}[section]
\newtheorem{corollary}{Corollary}[section]
\newtheorem{defin}{Definition}[section]
\newtheorem{example}{Example}[section]
\newtheorem{remark}{Remark}

\begin{document}
\author{Vincent RIVOIRARD\thanks{Laboratoire de Math\'ematique, CNRS 8628, Universit\'e Paris Sud, 91405 Orsay Cedex, France. D\'epartement de Math\'ematiques et Applications, ENS Paris, 45 rue d'Ulm, 75230 Paris Cedex 05, France. E-mail: Vincent.Rivoirard@math.u-psud.fr} \ and Judith ROUSSEAU \thanks{Ceremade,
Universit\'e Paris Dauphine,
Place du Mar\'echal de Lattre de Tassigny,
75016 Paris, France. CREST, 15, Boulevard Gabriel P\'eri
92245 Malakoff Cedex, France. E-mail: rousseau@ceremade.dauphine.fr}}
\title{Bernstein Von Mises Theorem for linear functionals of the density  }
%\runtitle{Bernstein Von Mises Theorem for  functionals of the density}
%\runauthor{V. Rivoirard and J. Rousseau.}

\maketitle

\begin{abstract}
In this paper, we study the asymptotic posterior distribution of linear functionals of the density. In particular, we give general conditions to obtain a semiparametric version of the Bernstein-Von Mises theorem. We then apply this general result to nonparametric priors based on infinite dimensional exponential families. As a byproduct, we also derive adaptive nonparametric rates of concentration of the posterior distributions under these families of priors on the class of Sobolev and Besov spaces. 
\end{abstract}


%%%%%%%%%%%%%%%%%%%%%%%%%%%%%%%%%%%%%%%%%%%%%%
%%%%%%%%%%%%%%%%%%%%%%%%%%%%%%%%%%%%%%%%%%%%%%
\section{Introduction}

The Bernstein-Von Mises property, in Bayesian analysis, concerns the asymptotic form of the posterior distribution of a quantity of interest, and more specifically it corresponds to the asymptotic normality of the posterior distribution centered at some kind of maximum likelihood estimator with variance being equal to the asymptotic frequentist variance of the centering point. Such results are well know in parametric frameworks, see for instance \cite{Lc} where general conditions are given.  This is an important property for both practical and theoretical reasons. In particular the asymptotic normality of the posterior distributions allows us to construct approximate credible regions and the duality between the behaviour of the posterior distribution and the frequentist distribution of the asymptotic centering point of the posterior implies that credible regions will have also good frequentist properties. These results are given in many Bayesian textbooks see for instance \cite{R} or \cite{B}.

In a frequentist perspective the Bernstein-Von Mises property enables the construction of confidence regions since under this property a Bayesian credible region will be asymptotically a frequentist confidence region as well. This is even more important in complex models, since in such models the construction of confidence regions can be difficult whereas, the Markov Chain Monte Carlo algorithms usually make the construction of a Bayesian credible region feasible. However the more complex the model the harder it is to derive Bernstein - Von Mises theorems.  In infinite dimensional setups, the mechanisms are even more complex.

Semi-parametric and non parametric  models are widely popular both from a theoretical and practical perspective and have been used by frequentists as well as Bayesians although their theoretical asymptotic properties have been mainly studied in the frequentist literature. The use of Bayesian non parametric or semi-parametric approaches is more recent and has been made possible mainly by the development of algorithms such as Markov Chain Monte-Carlo algorithms but has grown rapidly over the past decade.

However,  there is still little work on asymptotic properties of Bayesian procedures in semi-parametric models or even in nonparametric models. Most of existing works on the asymptotic posterior distributions deal with consistency or rates of concentration of the posterior. In other words it consists in controlling objects in the form
$
\Pp^\pi \left[ U_n |X^n\right]
$
where $\Pp^\pi[.|X^n]$ denotes the posterior distribution given a $n$ vector of observations$ X^n$ and $U_n$ denotes either a fixed neighbourhood (consistency) or a sequence of shrinking neighbourhoods (rates of concentration). As remarked by \cite{DF} consistency is an important condition since it is not possible to construct subjective prior in a nonparametric framework. Obtaining concentration rates of the posterior helps in understanding the impact of the choice of a specific prior and allows for a comparison between priors to some extent. However, to obtain a Bernstein-Von Mises theorem it is necessary not only to bound $\Pp^\pi \left[ U_n |X^n\right]$ but to determine an equivalent of $\Pp^\pi \left[ U_n |X^n\right]$ for some specific types of sets $U_n$. This difficulty explains that there is up to now very little work on Bernstein Von Mises theorems in infinite dimensional models. The most well known results are negative results and are given in \cite{F}. Some positiv
 e
 e results are provided by \cite{G00} on the  asymptotic normality of the posterior distribution of the parameter in an exponential family with increasing number of parameters. In a discrete setting \cite{BouGa} derive Bernstein-Von Mises results, in particular satisfied by Dirichlet priors. Nice positive results are obtained in \cite{KeL} and \cite{K}, however they rely heavily on a conjugacy type of property of the family of priors they consider and on the fact that their  priors put mass one on discrete probabilities which makes the comparison with the empirical distribution more tractable. 

In a semi-parametric framework, where the parameter can be separated into a parametric part, which is the parameter of interest and a non parametric part, which is the nuisance parameter, \cite{Ca} obtains interesting conditions leading to a Bernstein - Von Mises theorem on the parametric part, clarifying an earlier work of \cite{S}.

In this paper we are interested in studying the existence of a Bernstein-Von Mises property in semi-parametric models where the parameter of interest is a functional of the nuisance parameter, which is the density of the observations. The estimation of functionals of infinite dimensional parameters such as the cumulative distribution function at a specific point, is a widely studied problem both in the frequentist literature and in the Bayesian literature. There is a vast literature on the rates of convergence and on the asymptotic distribution of frequentist estimates of functionals of unknown curves and of finite dimensional functionals of curves in particular, see for instance \cite{VdV} for an excellent presentation of a general theory on such problems.

One of the most common functional considered in the literature is the cumulative distribution function calculated at a given point, say $F(x)$. The empirical cumulative distribution function, $F_n(x)$ is a natural frequentist estimator and its asymptotic distribution is   Gaussian with mean $F(x)$ and variance $F(x)(1-F(x))/n$.

The Bayesian counterpart of this estimator is the one derived from a Dirichlet process prior and it is well known to be asymptotically equivalent to $F_n(x)$, see for instance
\cite{GR}.%and to follow a Bernstein-von Mises theorem see cite{G} \textcolor{red}{Retrouver la ref de l'article du Ghosh (?) sur le BVM pour Dirichlet}.
 This result is obtained using the conjugate nature of the Dirichlet prior, leading to an explicit posterior distribution.
 Other frequentist estimators, based on frequentist estimates of the density have also been studied in the frequentist literature, in particular estimates based on kernel estimators. Hence a natural question arises. Can we generalize the Bernstein - Von Mises theorem of the Dirichlet estimator to other Bayesian estimators? What happens if the prior has support on distributions absolutely continuous with respect to Lebesgue Measure?

 In this paper we provide an answer to these questions by establishing conditions under which a Bernstein-Von Mises theorem can be obtained for linear functional of the density of $f$,  such as the cumulative distribution function $F(x)$, with centering its empirical counterpart, for instance $F_n(x)$ the empirical cumulative distribution function, when the prior puts positive mass on absolutely continuous densities with respect to Lebesgue measures. We also study cases where the asymptotic posterior distribution of the functional is not asymptotically Gaussian but is asymptotically a mixture of Gaussian distributions with different centering points.

%%%%%%%%%%%%%%%%%%%%%%%%%%%%%%%%%%%%%%%%%%%%%%%%%%%%%%%
\subsection{Notations  and aim }\label{subsec:notation}

In this paper, we assume that given a distribution $\Pp$ with a compactly supported density $f$ with respect to the Lebesgue measure, $X_1,...,X_n$ are independent and identically distributed by $\Pp$. We set $X^n = (X_1,...,X_n)$  and denote $F$ the cumulative distribution function associated with $f$. Without loss of generality we assume that for any $i$, $X_i \in [0,1]$ and we set $$\mathcal F = \left\{ f:[0,1] \rightarrow \R^+, \ \int_0^1f(x)dx =1\right\}.$$

We now define other notations that will be used throughout the paper. Denote $l_n(f)$ the log-likelihood associated with the density $f$ and if it is parametrized by a finite dimensional parameter $\theta$, $l_n(\theta) = l_n(f_\theta)$. For an integrable function $g$, we sometimes use the notation $F(g) = \int_0^1 f(u)g(u)du $. We denote by $<.,.>_f$ the inner product in
 $$\L_2(F)=\left\{g:\quad\int g^2(x)f(x)dx < + \infty \right\}$$ and by $|\!| . |\!|_f$ the corresponding norm.

We also consider the inner product in $\L_2[0,1]$ denoted $<.,.>_2$ and $|\!| . |\!|_2$ the corresponding norm. 
When there is no ambiguity we note $<.,.>_{f_0}$ by $<.,.>$ and $|\!|.|\!|_{f_0}$ by $|\!| . |\!|$.

Let $K(f,f')$ and $h(f,f') $ respectively the Kullback-Leibler divergence and 
  the Hellinger distance between two densities $f$ and $f'$, where we recall that
$$K(f,f') = F \left(\log(f/f')\right),\quad h(f,f') =\left[ \int \left(\sqrt{f(x)}-\sqrt{f'(x)}\right)^2dx\right]^{1/2},$$
and define $$V(f,f') = F\left((\log(f/f'))^2\right).$$ Finally, let $\Pp_0$ the true distribution of the observations $X_i$. $f_0$ is the associated density and $F_0$ the associated cumulative distribution function. We consider the usual notations on the empirical process, namely
 $$P_n(g)  = \frac{1}{n}\sum_{i=1}^ng(X_i),
\quad G_n(g)=\frac{1}{\rn}\sum_{i=1}^n[g(X_i)-F_0(g)],$$
and $F_n$  the empirical distribution function.

Consider a prior $\Pi$ on the set $\mathcal F$. The aim of this paper is to study the posterior distribution of $\Psi(f)$, where $\Psi$ is a continuous linear form on $\L_2[0,1]$ (a typical example is $\Psi(f) = F(x_0)=\Pp[X\leq x_0]$ for $x_0\in\R$) and to derive conditions under which
  \begin{eqnarray*}
  \Pp^\pi\left[ \rn (\Psi(f) - \Psi(P_n)) \leq z| X^n\right] \rightarrow  \Phi_{V_0}(z)  \quad \mbox{ in } \Pp_0 \mbox{ Probability},
  \end{eqnarray*}
 where $V_0$ is the variance of $\rn \Psi(P_n)$ under $\Pp_0$ and for any $V$, $\Phi_{V}(z)$ is the cumulative distribution function of a Gaussian random variable centered at 0 with variance $V$.

%%%%%%%%%%%%%%%%%%%%%%%%%%%%%%%%%%%%%%%%%
\subsection{Organization of the paper}

In Section \ref{sec:BVM} we present the general Bernstein Von Mises theorem, which is given in the formal way in the case where linear submodels are adapted to the prior. We then apply, in Section \ref{sec:BVMexp}, this general theorem to the  case where the prior is based on infinite dimensional exponential families. In this section, we first give general results giving the asymptotic posterior distribution of $\Psi(f)$ which can be either Gaussian or a mixture of Gaussian distributions. We also provide a theorem describing the posterior concentration rate under such priors (see Section \ref{subsec:rate}). Finally, in Section \ref{subsec:ex:cdf}, using an example, we explain how bad phenomenons can occur. The proofs are postponed in Section \ref{sec:proofs}.

%%%%%%%%%%%%%%%%%%%%%%%%%%%%%%%%%%%%%%%%%%%%%%%%%%%%%%%%%
%%%%%%%%%%%%%%%%%%%%%%%%%%%%%%%%%%%%%%%%%%%%%%%%%%%%%%%%%
\section{Bernstein Von Mises theorems} \label{sec:BVM}
%%%%%%%%%%%%%%%%%%%%%%%%%%%%%%%%%%%%%%%%%%%%%%%%%%%%%%%
\subsection{Some heuristics for proving Bernstein Von Mises theorems}
We first define some notions that are useful in the study of asymptotic properties of semi-parametric models. These notions can be found for instance in \cite{VdV}.

As in Chapter 25 of \cite{VdV},  to study the asymptotic behaviour of semi-parametric models we consider 1-dimensional differentiable paths locally around the true parameter $f_0$,  that is submodels of the form:
 $u \rightarrow f_u $ for $0<u<u_0$, for some $u_0>0$ such that for each path there exists a measurable function $g$ called the score function for the submodel $\{f_u, , 0<u< u_0\}$ at $u=0$ satisfying
\begin{eqnarray} \label{def:score}
\lim_{u\rightarrow 0}\int_\R \left(\frac{f_u^{1/2}(x)-f_0^{1/2}(x)}{u} -\frac{1}{2}g(x)f_0^{1/2}(x)\right)^2dx = 0.
\end{eqnarray}
We denote by $\mathcal F_{f_0}$ the  tangent set, i.e. the collection of score
functions   $g$   associated    with   these   differentiable   paths.  Using
(\ref{def:score}), $\mathcal F_{f_0}$ can be identified
with a subset of $\{g\in\L_2(F_0):\ F_0(g)=0\}$.
For instance,  when  considering  all  probability
laws, the most usual collection of differentiable paths is given by
\begin{equation}\label{score}
f_u(x) = c(t)f_0(x) e ^{ug(x) }
\end{equation}
with $\norm{g}_\infty<\infty$ and $c$ such that
$c(0)=1$ and  $c'(0)=0$.  In  this case,  $g$  is the
score function. Note that as explained in \cite{VdV}, the collection of differentiable paths of the form $f_u(x) = 2c(u)f_0(x)(1+\exp(-2ug(x)))^{-1}$ (with previous conditions on $c$),  leads to the tangent space given by $\{g\in\L_2(F_0):\ F_0(g)=0\}$.

Now, consider a continuous linear form  $\Psi$ on $\L_2$. We can identify such a functional by a function $ \psi\in  \L_2$ such that  for all $f\in \L_2$
\begin{equation}\label{Riesz}
\Psi(f) = \int f(x)\psi(x)dx.
\end{equation} Then for any differentiable path $t \rightarrow f_t$ with score
function $g$, if the function $\psi $ is bounded on $\R$ (or on the support
of $f_u$ for all $0\leq u<u_0$),
\begin{eqnarray*}
\frac{\Psi(f_u)  -\Psi(f_0)}{t} &=&  \int \psi (x)    g(x)f_0 (x)dx + \int \frac{\left(f_u^{1/2}(x)-f_0^{1/2}(x)\right)^{2}}{u}\psi(x) dx \\
 & & + 2\int \psi(x)
 \left(\frac{f_u^{1/2}(x)-f_0^{1/2}(x)}{u} -\frac{1}{2}g(x)f_0^{1/2}(x)\right)f_0^{1/2}(x)dx \\
  &=&<\psi,g> + o(1).
\end{eqnarray*}
Then, we  can define the efficient influence function $\tilde{\psi}$ belonging
to $\overline{\mbox{lin}}(\mathcal F_{f_0})$ (the closure of the linear space generated by $\mathcal F_{f_0}$) that satisfies for any $g \in \mathcal F_{f_0},$
$$\int  \tilde{\psi}(u)  g(x)f_0(x)dx =  \int  \psi(x)  g(x)f_0 (x)dx.$$  This
implies:
\begin{equation}\label{limPsi}
\lim_{u\to 0}\frac{\Psi(f_u) -\Psi(f_0)}{u}=<\tilde\psi,g>.
\end{equation}
The efficient influence function  will  play an  important  role for  our purpose.  The efficient
influence function is
also a key notion to characterize asymptotically efficient estimators (see Section
25.3 of \cite{VdV}).

Now, let us provide some examples by specifying different types of continuous linear forms that can be considered.
\begin{example}
An important example is provided by the cumulative distribution function. If $x_0\in\R$ is fixed, consider for any density function $f\in \L_2$ whose cdf is $F$,
$$\Psi(f)=\int \1_{x\leq x_0} f(x) dx=F(x_0)$$ so that in this case, $\psi(u) = \1_{x \leq x_0}$, which is a bounded function  and if $\mathcal F_{f_0}$ is the subspace of $\L_2(F_0)$ of functions $g$ satisfying $F_0(g)=0$ then $\tilde{\psi}(x) = \1_{x\leq x_0} - F_0(x_0)$.
\end{example}
\begin{example}
 More generally, for any measurable set $A$  consider $\psi(x) = \1_{x \in A}$ and for any density function $f\in \L_2$
$$\Psi(f)=\int\1_{x \in A}f(u)du$$satisfies the above conditions and $\tilde{\psi}(x) = \1_{x \in A}- \int_A f_0(x)dx$.
\end{example}
\begin{example}
If $f_0 $ has bounded support, say on $[0,1]$ then the functional $$\Psi(f)= \E_f[X] = \int_0^1 xf(x)dx$$ satisfies the above conditions, $\psi(x) =x$ and $\tilde{\psi}(x) = x- \E_{f_0}[X]$.
\end{example}

In this framework, the Bernstein Von Mises theorem could be derived from the convergence of the following Laplace transform defined for any $t\in\R$ by
\begin{eqnarray*}
L_n(t) &=& \esp^\pi[\exp(t\rn(\Psi(f)-\Psi(P_n)))| X^n ] \\
 &=& \frac{\int\exp\left(t\rn(\Psi(f)-\Psi(P_n))+l_n(f)-l_n(f_0)\right) d\pi(f) }{\int\exp\left(l_n(f)-l_n(f_0)\right) d\pi(f) }.
\end{eqnarray*}
Now, let us set $f_{g,n}=f_u$ if $u=n^{-\frac{1}{2}}$. We have:
\begin{eqnarray*}
\sqrt{n}\left(\Psi(f_{g,n})-\Psi(P_n)\right)&=&\sqrt{n}\int\psi(x)(f_{g,n}(x)-f_0(x))dx-G_n(\tilde\psi)\\
&=&<\tilde\psi,g>  - G_n(\tilde\psi) +\Delta_n(g).
\end{eqnarray*}
%with
%\[\Delta_n(g) =\rn \left(\Psi(f_{g,n})-\Psi(f_0)\right) - <\tilde\psi,g>.\]
Furthermore,
\begin{eqnarray*}
l_n(f_{g,n}) - l_n(f_0)
&=& R_n(g)+G_n(g)-\frac{F_0(g^2)}{2},
\end{eqnarray*}
with
\[R_n(g) = nP_n\left(\log\left( \frac{f_{g,n}}{f_0} \right)\right) - G_n(g) + \frac{F_0(g^2)}{2}.\]
So,
\begin{eqnarray*}
&&t\sqrt{n}\left(\Psi(f_{g,n})-\Psi(P_n)\right)+l_n(f_{g,n}) - l_n(f_0)\\
&&\hspace{3cm}=R_n(g)-\frac{F_0(g^2)}{2}+G_n(g-t\tilde\psi)+t\Delta_n(g)+t<\tilde\psi,g>\\
&&\hspace{3cm}=R_n(g-t\tilde\psi)+G_n(g-t\tilde\psi)-\frac{F_0((g-t\tilde\psi)^2)}{2}+\frac{t^2F_0(\tilde\psi^2)}{2}+U_n,
\end{eqnarray*}
with
\[U_n=t\Delta_n(g) + R_n(g)-R_n(g-t\tilde\psi).\]
Lemma 25.14 of \cite{VdV} shows that under (\ref{def:score}), $R_n(g)=o(1)$ and (\ref{limPsi}) yields $\Delta_n(g)=o(1)$ for a fixed $g$.
It is not enough however to derive a Bernstein-Von Mises theorem. Nonetheless if we can
 choose a prior distribution $\pi$ adapted to the previous framework to obtain uniformly $U_n=o(1)$,
\[\sqrt{n}\left(\Psi(f_{g,n})-\Psi(f)\right)+l_n(f_{g,n}) - l_n(f)=o(1)\]
 and the equalities
\begin{eqnarray*}\label{cond:big}
\frac{\int e^{R_n(g-t\tilde\psi)+G_n(g-t\tilde\psi)-\frac{F_0((g-t\tilde\psi)^2)}{2}}d\pi(f)}{\int e^{R_n(g)+G_n(g)-\frac{F_0(g^2)}{2}}d\pi(f)}&=&\frac{\int \exp\left(l_n(f)-l_n(f_0)\right) d\pi(f_{g+t\tilde\psi}) }{\int\exp\left(l_n(f)-l_n(f_0)\right)  d\pi(f) }\\
&=&1+o(1),
 \end{eqnarray*}
then
\[L_n(t)=\exp\left(\frac{t^2F_0(\tilde\psi^2)}{2}\right)(1+o(1)).\]
In this case, our goal is reached. However, it is not obvious that a given prior $\pi$ satisfies all these properties. In particular, in a nonparametric framework, the property $R_n(g)= o(1)$ uniformly over a set whose posterior probability goes to 1, is usually not satisfied.
We thus consider an alternative approach based on linear submodels.

%%%%%%%%%%%%%%%%%%%%%%%%%%%%%%%%%%%%%%%%%%%%%%%%%%%%%%%%%%%%%%%%%%%%%%%%%%%%%%
\subsection{Bernstein Von Mises under linear submodels} \label{subsec:BVM:line}
In this section we study the case where linear local models are
adapted to the prior. More precisely, we assume that $\norm{\log(f_0)}_\infty<\infty$ so, for each density function $f$, we
define $h$ such that for any $x$, $$h(x)=\sqrt{n}\log\left(\frac{f(x)}{f_0(x)}\right)\quad\mbox{or equivalently}\quad f(x) = f_0(x)\exp\left(\frac{h(x)}{\sqrt{n}}\right).$$ For the sake of clarity, we sometime write $f_h$ instead of $f$ and  $h_f$ instead of $h$ to underline the relationship between $f$ and $h$.
Note that in this context $h$ is not the score function since $F_0(h)\neq 0$.
It would be equivalent
 to consider local models of the form  $f = f_0(1+h/\rn)$, except that
we would have to impose constraints on $h$ for $f$ to be positive.
We consider a continuous linear form  $\Psi$ on $\L_2$ such that for any $f\in \L_2$, we consider $\psi$ such that (\ref{Riesz}) is satisfied and we set for any $x$,
\begin{equation}\label{psic}
\psi_c(x) =\psi(x)-F_0(\psi).
\end{equation}
Note that  $\psi_c$ coincides with the influence function $\tilde\psi$ associated with the tangent set $\{g \in \L_2(F_0); F_0(g) = 0\}$. Then we consider the following assumptions.
\begin{itemize}
\item[(A1)]  The posterior distribution concentrates around $f_0$. More precisely, there exists $u_n=o(1)$ such that if
$A_{u_n}^1 = \left\{f \in
\mathcal F:\quad V(f_0,f) \leq u_n^2\right\}$
the posterior distribution of $A_{u_n}^1$ satisfies
$$\P^\pi\left\{ A_{u_n}^1|X^n\right\} = 1 + o_{\P_0}(1).$$
\item[(A2)] The posterior distribution of the subset $A_{n} \subset A_{u_n}^1$  of densities such that
\begin{equation}\label{ineq:A2}
\int\left|\log\left(\frac{f(x)}{f_0(x)}\right)\right|^3\left(f_0(x)+f(x)\right)dx=o(1)
\end{equation}
satisfies
$$\P^\pi\left[ A_{n}|X^n\right] = 1 + o_{\P_0}(1).$$
\item[(A3)]
Let $$R_n(h)  =\sqrt{n}F_0(h) + \frac{F_0(h^2)}{2} $$ and for any $x$, $$\bar{\psi}_{t,n}(x) =\psi_c(x) + \frac{\rn}{t} \log\left(
F_0\left[\exp\left(\frac{h}{\rn} - \frac{t\psi_c}{\rn}\right)\right]\right).$$ We have
\begin{eqnarray} \label{cond:big:2}
 && \frac{\int_{A_{n}}\exp\left(-\frac{F_0((h_f-t\bar{\psi}_{t,n})^2)}{2}+ \G_n(h_f-t\bar{\psi}_{t,n}) +
R_n(h_f-t\bar{\psi}_{t,n})\right) d\pi(f) }{\int_{A_{n}}\exp\left(-
\frac{F_0(h_f^2)}{2}+ G_n(h_f) +  R_n(h_f)\right)d\pi(f) } \nonumber \\
 &  &\quad \quad \quad =  1+ o_{\Pp_0}(1).
 \end{eqnarray}
\end{itemize}
Before stating our main result, let us discuss these assumptions.
Condition (A1) concerns concentration rates of the posterior distribution and
there exists now a large literature on such results. See
for instance \cite{SW} or \cite{GGVdV} for general results. The difficulty here comes
from the use of $V$ instead of the Hellinger or
the $\L_1$-distance. However since $u_n$ does not need  to be
optimal, deriving rates in terms of  $V$ from those in terms  of the
Hellinger distance is often not a problem (see below).

Condition (A2) is a refinement of (A1) but can often be derived from (A1) as illustrated below.

The main difficulty comes from condition (A3). To prove
it, we need to be able to construct a transformation
$T$ such that $Tf_h = f_{h-t\bar{\psi}_{t,n}}$ exists and such that the
prior is hardly modified by this transformation. In parametric
setups, continuity of the prior near the true value is enough to
ensure that the prior would hardly be modified by such a transform and
this remains true in semi-parametric setups where we can write
the parameter as $(\theta, \eta)$ where $\theta $ is the parameter
of interest and is finite dimensional. Indeed as shown in \cite{Ca}
under certain conditions the transformations can be transferred to
transformations on $\theta$ which is finite dimensional. Here this
 is more complex since $T$ is a transformation on $f$ which is
infinite dimensional so that a condition of the form $d\pi(Tf) =
d\pi(f)(1+ o(1))$ does not necessarily make sense.  We study this
aspect in more details in Section \ref{sec:BVMexp}.

Now, we can state the main result of this section.
\begin{thm} \label{th:sieve:1}
Let $f_0$ be a density on $\mathcal F$ such that $\norm{\log(f_0)}_\infty<\infty$  and $\norm{\psi}_\infty<\infty$. Assume that
(A1), (A2) and (A3) are true. Then, if
$$\Psi(P_n) = P_n(\psi) = \frac{\sum_{i=1}^n\psi(X_i)}{n}$$ we have for any $z$, in probability with respect to $ \Pp_0$,
$$\Pp^\pi\left\{\sqrt{n}(\Psi(f)-\Psi(P_n))\leq z|X^n\right\}
- \Phi_{F_0(\psi_c^2)}(z) \rightarrow 0.$$
\end{thm}

The proof of Theorem \ref{th:sieve:1} is given in Section \ref{sec:th:sieve:1}.

Sieve priors lead to interesting behaviours of the posterior
distribution as illustrated in the following section. Indeed they
have a behaviour which is half way between parametric and non
parametric. We illustrate these
features in the following two sections.
%%%%%%%%%%%%%%%%%%%%%%%%%%%%%%%%%%%%%%%%%%%%%%%%%%%%%%%%%%%%%%%
%%%%%%%%%%%%%%%%%%%%%%%%%%%%%%%%%%%%%%%%%%%%%%%%%%%%%%%%%%%%%%%
\section{ Bernstein Von Mises theorem under infinite dimensional exponential families} \label{sec:BVMexp}
In this section, we study a specific class of priors based on infinite dimensional exponential families on the following class of densities supported by $[0,1]$:
$$\mathcal F =\left\{f \geq 0: \quad  f \mbox{ is 1-periodic}, \ \int_0^1 f(x)dx =1, \ \log(f) \in \L_2([0,1])\right\}.$$
We  assume that  $f_0\in{\mathcal F}$ and we consider two types of orthonormal bases defined in the following section, namely the Fourier and wavelet bases.
%%%%%%%%%%%%%%%%%%%%%%%%%%%%%%%%%%%%%%%%%%%%%%%%
\subsection{Orthonormal bases}\label{orthobasis}

Fourier bases constitute unconditional bases of periodized Sobolev spaces $W^\ga$ where $\gamma$ is the smoothness parameter. Our results are also valid for a wide range of Besov spaces. In this case, we consider wavelet bases which allow for the following expansions:
$$f(x)=\th_{-10}\1_{[0,1]}(x)+\sum_{j=0}^{+\infty}\sum_{k=0}^{2^j-1}\th_{jk}\varPsi_{jk}(x),\quad x\in [0,1]$$
where $\th_{-10}=\int_0^1 f(x)dx$ and $\th_{jk}=\int_0^1 f(x)\varPsi_{jk}(x)dx.$ We recall that the functions $\varPsi_{jk}$ are obtained by periodizing dilations and translations of a mother wavelet $\varPsi$ that can be assumed to be supported by the compact set $[-A ,A]$: $$\varPsi_{jk}(x)=2^{\frac{j}{2}}\sum_{l=-\infty}^{+\infty}\varPsi(2^jx-k+2^jl),\quad x\in [0,1].$$
If $\varPsi$ belongs to the H\"older space $C^r$ and has $r$ vanishing moments then the wavelet basis constitutes an unconditional basis of the Besov space ${\mathcal B}^{\ga}_{p,q}$ for $1\leq p,q\leq +\infty$ and $\max\left(0,\frac{1}{p}-\frac{1}{2}\right)<\ga<r$. In this case, ${\mathcal B}^{\ga}_{p,q}$ is the set of functions $f$ of $\L_2[0,1]$ such that $\norm{f}_{\ga,p,q}<\infty$ where
$$\norm{f}_{\ga,p,q}=\left\{\begin{array}{ll}|\th_{-10}|+\left(\sum_{j=0}^{+\infty}2^{jq(\ga+\frac{1}{2}-\frac{1}{p})}\left(\sum_{k=0}^{2^j-1}|\th_{jk}|^p\right)^{\frac{q}{p}}\right)^{\frac{1}{q}}&\mbox{ if }q<\infty\\
|\th_{-10}|+\sup_{j\geq 0}\left\{2^{j(\ga+\frac{1}{2}-\frac{1}{p})}\left(\sum_{k=0}^{2^j-1}|\th_{jk}|^p\right)^{\frac{1}{p}}\right\}&\mbox{ if }q=\infty.
\end{array}\right.$$
We refer the reader to \cite{Mey} for a good review of wavelets and Besov spaces. We just mention that Besov spaces include in particular Sobolev spaces ($W^\ga={\mathcal B}^{\ga}_{2,2}$) and, when $\ga$ is not an integer, H\"older spaces
($C^\ga={\mathcal B}^{\ga}_{\infty,\infty}$).
To shorten notations, the orthonormal basis will be denoted $(\phi_{\la})_{\la\in \Ne}$, where $\phi_0=\1_{[0,1]}$ and
\begin{itemize}
\item[-] for the Fourier basis, for $\la\geq 1$, $$\phi_{2\la-1}(x)=\sqrt{2}\sin(2\pi\la x),\quad  \phi_{2\la}(x)=\sqrt{2}\cos(2\pi\la x).$$
 \item[-] for the wavelet basis, if $\la=2^j+k$, with $j\in\Ne$ and $k\in\{0,\dots,2^j-1\}$,
$$\phi_{\la}=\varPsi_{jk}.$$
\end{itemize}
Now, the decomposition of each periodized function $f\in \L_2[0,1]$ on  $(\phi_{\la})_{\la\in \Ne}$ is written as follows: $$f(x)=\sum_{\la\in\Ne}\th_\la\phi_{\la}(x),\quad x\in [0,1],$$
where $\th_\la=\int_0^1f(x)\phi_\la(x)dx$. Recall that when the Fourier basis is used, $f$ lies in $W^\ga$ for $\ga>0$ if and only if $\norm{f}_{\ga}<\infty$, where
$$\norm{f}_{\ga}=\left(\th_0^{2}+\sum_{\la\in\Ne^*}|\la|^{2\ga}\th_\la^{2}\right)^{\frac{1}{2}}.$$
We respectively use  $\norm{.}_{\ga}$ and $\norm{.}_{\ga,p,q}$ to define the radius of the balls of $W^\ga$ and ${\mathcal B}^{\ga}_{p,q}$ respectively.
We now present the general result on posterior concentration rates associated with such prior models.
%%%%%%%%%%%%%%%%%%%%%%%%%%%%%%%%%%%%%%%%%%%%%%%
\subsection{Posterior rates} \label{subsec:rate}
Assume that $f_0 \in \mathcal F$ and  let $\varPhi$ be one of the orthonormal basis introduced in Section \ref{orthobasis}, then
$$\log(f_0)-\int_0^1\log(f_0(x))dx=\sum_{\la\in\Ne^*}\thol\phi_{\la}.$$
Set $\theta_0 = ( \theta_{0\lambda})_{\lambda \in \N^*}$ and define
$c(\theta_0)=-\int_0^1\log (f_0(x) ) dx,$ we have
$$ f_0(x)=\exp\left(\sum_{\la\in\Ne^*}\thol\phi_{\la}(x)-c(\theta_0)\right).
$$
We consider the following family of models: for any $k\in\Ne^*$, we set
 \begin{eqnarray*}
 \mathcal F_k =\left\{ f_\th=\exp\left(\sum_{\la=1}^k\thl\phi_{\la}-c(\th)\right):\quad \theta \in \R^k  \right\},
 \end{eqnarray*}
where
\begin{eqnarray}\label{def:c}
c(\th)=\log\left(\int_0^1\exp\left(\sum_{\la=1}^k\thl\phi_{\la}(x)\right)dx\right).\end{eqnarray}
So, we define a prior $\pi$ on the set $\mathcal F = \cup_k \mathcal F_k$ by defining a prior $p$ on $\Ne^*$ and then, once $k$ is chosen, we fix a prior $\pi_k$ on $\mathcal F_k$.
%$$d\pi(f_\theta ) = p(k)\pi_k(\theta)d\theta, \quad \theta =(\theta_1,...,\theta_k).$$ Note that a sieve model with deterministic $k_n$ in the form $f_\theta (y) = \exp\{\sum_{j=1}^{k_n}\theta_j\phi_j(y) -c(\theta) \},$ corresponds to the case $p(k) = \delta_{k=k_n}$.Another possible version of this prior is to consider $f_\theta (y)  =\exp\{\sum_{j=1}^\infty \theta_j\phi_j(y) -c(\theta) \}$.
Such priors are often considered in the Bayesian non parametric literature. See for instance \cite{Scri}. The special case of log-spline priors has been studied by \cite{GGVdV} and \cite{Huang}, whereas the prior considered by \cite{VW} is based on Legendre polynomials. For the wavelet case, \cite{Huang} considered the special case of the Haar basis.

Since one of the key conditions needed to obtain a Bernstein Von Mises theorem is a concentration rate of the posterior distribution of order $\epsilon_n$, we first give two general results on concentration rates of posterior distributions based on the two different setups of orthonormal bases: the  Fourier basis and the  wavelet basis. These results have their own interest since we obtain  in such contexts optimal adaptive rates of convergence. In a similar spirit \cite{Scri} considers infinite dimensional exponential families and derives minimax and adaptive posterior concentration rates. Her work differs from the following theorem in two main aspects. Firstly she restricts her attention to the case of Sobolev spaces and Fourier basis, whereas we consider Besov spaces and secondly she obtains adaptivity by putting a prior on the smoothness of the Sobolev class whereas we obtain adaptivity by constructing a prior on the size $k$ of the parametric spaces, which to our opinion
  is a more natural approach. Moreover \cite{Scri} merely considers Gaussian priors. Also related to this problem is the work of \cite{Huang} who derives a general framework to obtain adaptive posterior concentration rates and apply her results to the Haar basis case. The limitation in her case, apart from the fact that she considers the Haar basis and no other wavelet basis is that she constraints the $\theta_j$'s in each $k$ dimensional model to belong to a ball with fixed radius.

Now, we specify the conditions on the  prior $\pi$:
\begin{defin}\label{defin:prior}
Let $1>\be> 1/2$ be fixed and let
$g$ be  a  continuous and positive density on $\R$  bounded (up to a contant) by the function $M_{p_*}(x)=\exp\left(-c|x|^{p_*}\right)$ for  positive constants $c,p_*$ and assume that
for all $M>0$ there exists $a,b$ such that
$$g(y+u) \geq a\exp \{ -b (|y|^{p_*} + |u|^{p_*})\}, \quad \forall |y|\leq M, \quad \forall u \in \R$$
The prior $p$ on $k$ satisfies one of the following conditions:\\\\
{\bf [Case (PH)]} \ There exist two positive constants $c_1$ and $c_2$ such that for any $k\in\Ne^*$,
\begin{equation}\label{assprior}
\exp\left(-c_1 k L(k)\right)\leq p(k) \leq \exp\left(-c_2 k L(k)\right),
\end{equation}
where $L$ is the function  that can be either $L(x)=1$ or $L(x)=\log(x)$.\\\\
{\bf [Case (D)]} \  If $k_n^*=n^{1/(2\beta+1)}$,
$$p(k) = \delta_{k_n^*}(k).$$
Conditionally on $k$ we define the prior on $\mathcal F_k$ by assuming that the prior distribution $\pi_k$ on $\theta = (\theta_\lambda)_{1 \leq \lambda \leq k}$ is given by
$$\frac{\thl}{\sqrt{\tau_\la}} \sim g, \quad  \tau_\la = \tau_0 \la^{-2\beta} \quad \mbox{i.i.d.}$$
where $\beta < 1/2 +p_*/2$ if $p_*\leq 2$ and $\beta< 1/2 + 1/p_*$ if $p_*>2$.
\end{defin}
Observe that we do not necessarily consider Gaussian priors since we allow for densities $g$ to have different tails. The prior on $k$ can be non random, which corresponds to the Dirac case (D).
For the case (PH), $L(x) = \log(x)$ corresponds typically to a Poisson prior on $k$ and the case $L(x)=1$ corresponds typically to hypergeometric priors. Now, we have the the following result.
%%%%%%%%%%%%
\begin{thm}\label{th:post:rate}
Assume that $\norm{\log(f_0)}_{\infty}<\infty$ and that there exists $\gamma >1/2$ such that  $\log(f_0)  \in {\mathcal B}^\ga_{p,q}$, with $p\geq 2$ and $1\leq q\leq \infty$. Then,
\begin{eqnarray} \label{post:rate:0}
\Pp^\pi\left\{ f_\theta: \quad h(f_0,f_\theta) \leq \frac{\log n}{L(n)}\ep_n | X^n\right\}=1+o_P(1),
\end{eqnarray}
and
\begin{eqnarray} \label{post:rate:1}
\Pp^\pi\left\{ f_\theta: \quad|\!|\theta_0 - \theta|\!|_2 \leq \frac{\left(\log n\right)^2}{L(n)} \ep_n | X^n\right]=1+o_P(1),
\end{eqnarray}
where in case (PH),
$$\ep_n=\ep_0\left(\frac{\log n}{n}\right)^{\frac{\gamma}{2\gamma+1}},$$
in case (D),
$$\ep_n=\ep_0\log n n^{-\frac{\be}{2\be+1}}, \quad \mbox{if } \gamma \geq \beta$$
$$\ep_n=\ep_0n^{-\frac{\gamma}{2\be+1}}, \quad \mbox{if } \gamma < \beta$$
and  $\ep_0$ is a constant large enough.
\end{thm}

The proof of Theorem \ref{th:post:rate} is given in Section \ref{sec:th:post:rate}.
%%%%%%%%
\begin{remark}\label{tailp}
If the density $g$ only  satisfies a tail condition of the form
 $$g(x) \leq C_g |x|^{-p_*}, \quad |x| \quad \mbox{large enough}  $$
with $p_*>1$, then, in case (PH), if $\gamma >1$   the rates defined by (\ref{post:rate:0}) and (\ref{post:rate:1}) remain valid.
\end{remark}

\begin{remark}\label{rem:rate}
 Note that in the case (PH) the posterior concentration is, up to a $\log n$ term, the minimax rate of convergence on the  collection of spaces with smoothness $\gamma>1/2$, whereas in the case (D) the minimax rate is achieved only when $\gamma =\beta$. \end{remark}
%%%%%%%%%%%%%%%%%%%%%%%%%%%%%%%%%%%%%%%%%%%%%%%%%
\subsection{Bernstein Von Mises under these models} \label{subsec:BVM:BON}
In this section, we  apply Theorem \ref{th:sieve:1} of Section \ref{subsec:BVM:line} to establish the following Bernstein Von Mises-type result. For this purpose, let us expand the function $\psi_c$ defined in (\ref{psic}) on the basis $(\phi_{\la})_{\la\in \Ne}$:
$$\psi_c=\sum_{\la\in\Ne}\psi_{c,\la}\phi_\la.$$ We denote $\Pi_{f_0,k}$ the projection operator on the vector space generated by  $(\phi_{\la})_{0\leq\la\leq k}$ for the scalar product $<f,g>=F_0(fg)$ and $\Delta_\psi = \psi_c - \Pi_{f_0,k}\psi_c$. So we can write for any $x\in [0,1]$,
\[\Pi_{f_0,k}\psi_c(x)=\psi_{\Pi,c,0}+\sum_{\la=1}^k\psi_{\Pi,c,\la}\phi_\la(x),\]
since $\phi_0(x)=1.$ We denote $B_{n,k}$ the renormalized sequence of coefficients that appear in the above sum: $$B_{n,k}=\frac{\psi_{\Pi,c,[k]} }{\rn},\quad\psi_{\Pi,c,[k]}=(\psi_{\Pi,c,\la})_{1\leq\la\leq k}.$$ Such quantities will play a key role in the sequel. Let $l_0>0$ be large enough  so that
$$\Pp^\pi\left[ k>\frac{l_0 n\epsilon_n^2}{L(n)} |X^n\right] \leq e^{-cn\epsilon_n^2},$$
for some positive $c>0$, where $\epsilon_n$ is the posterior concentration rate defined in Theorem \ref{th:post:rate} and define $l_n = l_0 n\epsilon_n^2/L(n)$. In the case (D) we set $l_n = k_n^*$. In the following, in the case (D), whenever a statement concerns $k \leq l_n$ it is to be understood as $k=l_n$.

We have the following result.
%%%
\begin{thm} \label{th:expo:BVM}
Let us assume that the prior is defined as in Definition \ref{defin:prior} and for all $t\in\R$,  $1\leq k\leq l_n$ (or $k_n^*$ in case (D)), assume that 
 \begin{eqnarray} \label{cond:prior}
\frac{ \pi_k(\theta)}{\pi_k(\theta-tB_{n,k})} = 1 + o(1), \quad \mbox{if } \quad \sum_{j =1}^k(\theta_j-\theta_{0j})^2 \leq  \frac{\left(\log n\right)^2}{L(n)} \ep_n
\end{eqnarray}
 uniformly over $\{\theta; |\!| \theta -\theta_0|\!|_2 \leq 3(\log n)^2 \epsilon_n\}$. 
Assume also that 
\begin{eqnarray}\label{cond:psij}
\sup_{k \leq l_n}( |\!|\sum_{j>k} \psi_{cj}\phi_j |\!|_\infty + \sqrt{k}|\!| \sum_{j>k} \psi_{cj}\phi_j|\!|_2 ) = o\left( \frac{(\log n)^{-2}}{\sqrt{n}\epsilon_n^2}\right).
\end{eqnarray}
 Under assumptions of Theorem \ref{th:post:rate}
\begin{itemize}
\item  for all $z \in \R$
\begin{eqnarray} \label{resu:nonBVM}
\Pp^\pi\left[ \rn(\Psi(f) - \Psi(P_n) ) \leq z  |X^n\right] &=&
\sum_{k} p(k|X^n)\Phi_{V_{0k}} \left(z+\mu_{n,k}\right)+o_{\Pp_0}(1), \nonumber \\
\end{eqnarray}
where
\begin{itemize}
\item $V_{0k} = F_0(\psi^2_c)-F_0(\Delta_\psi^2)$.
\item $\mu_{n,k}= -\rn F_0[(\psi_c-\Pi_{f_0,k}\psi_c )\sum_{j\geq k+1}\theta_{0j}\phi_j]+ \G_n(\Delta_\psi)$
\end{itemize}
\item  In the case (D), if $\gamma\geq \beta$,
\begin{eqnarray} \label{BVM}
\Pp^\pi\left[ \rn(\Psi(f) - \Psi(P_n) ) \leq z  |X^n\right] &=&
\Phi_{V_0}\left(z\right)+o_{\Pp_0}(1),%\nonumber \\
\end{eqnarray}
where $V_{0} = F_0(\psi^2_c)$.
\end{itemize}
\end{thm}
The Bernstein-Von Mises property obtained in the case (D) is deduced by proving relation (\ref{cond:psij}) in this case. Indeed if $\gamma \geq \beta$ there exists $a>0$ such that $\sqrt{n} \epsilon_n^2 \leq n^{-a}$, besides 
\begin{eqnarray*}
 |\!| \sum_{j >k} \psi_{cj} \phi_j |\!|_\infty &\leq& 1 + |\!| \sum_{j \leq k} \psi_{cj} \phi_j |\!|_\infty \\
  &\leq& C \log n
 \end{eqnarray*}
 and $\sum_{j\geq k+1} \psi_{cj}^2 = 0(1/k)$, so that relation  (\ref{cond:psij}) is satisfied. Apart from this argument the proof of Theorem \ref{th:expo:BVM} is given in Section \ref{sec:th:expo:BVM}.
%%%%%
The first part of Theorem \ref{th:expo:BVM} shows that the posterior
distribution of $\rn(\Psi(f) - \Psi(P_n))$ is asymptotically a
mixture of Gaussian distributions with variances $V_0 - F_0(\Delta_\psi^2)$ and
 mean values $\mu_{n,k}$ with weight $p(k|X^n)$. To
obtain an asymptotic Gaussian distribution with mean zero and variance $V_0$ it is
necessary for $\mu_{n,k}$ to be small whenever $p(k|X^n)$ is not.
The conditions given in the second part of Theorem \ref{th:expo:BVM}
ensure that this is the case, however they are not necessary
conditions. Nevertheless, in Section \ref{subsec:ex:cdf}, we give a counter-example for which the Bernstein-Von Mises property is not satisfied in the cases (PH) and (D) with $\gamma<\beta$.  

We now discuss condition (\ref{cond:prior}) in three different examples. Note first that $A_n\subset \{\theta; |\!|\theta -\theta_0 |\!|_2\leq 3 (\log n)^2 \epsilon_n\}$ with $\theta \in \Theta_k$, $k \leq l_n$. 
\begin{itemize}
\item Gaussian: If $g$ is Gaussian  then for all $k \leq l_n$ (or $k_n^*$ in the case of a type (D) prior) and all $j \leq k$, $\theta_j\sim
\mathcal N( 0, \tau_0^2j^{-2\beta})$ and for  all $\theta \in A_n \cap \mathcal F_k$
\begin{eqnarray*}
  \frac{\sum_{j=1}^k {\psi_c}_j^2j^{2\beta}}{n} &\leq& \frac{Ck^{2\beta} }{n} \leq O(n^{2\beta -1} \epsilon_n^{4\beta}) = o(1) \\
  \frac{ \sum_{j=1}^k\theta_j {\psi_c}_jj^{2\beta }}{\rn } &=& \frac{\sum_{j=1}^k(\theta_j-\theta_{0j}) {\psi_c}_jj^{2\beta }+ \sum_{j=1}^k\theta_{0j} {\psi_c}_jj^{2\beta }}{\rn} \\
   &\leq&
 \frac{C}{\rn} \left[  |\!|\theta -\theta_0 |\!|k^{2\beta} +  (k^{2\beta-\gamma} +1) \right]  \\
 &=&  o(1).
\end{eqnarray*}
 This implies that uniformly over
$A_n $
$$\pi_k(\theta-B_{n,k}) = \pi_k(\theta)(1+o(1))$$

\item Laplace: If $g$ is Laplace, $g(x) \propto e^{-|x|}$,
\begin{eqnarray*}
\left|\log\left(g\left(\frac{\theta_j - t{\psi_c}_j/\rn}{\sqrt{\tau_j}}\right) \right) -
\log\left( g\left(\frac{\theta_j}{\sqrt{\tau_j}}\right)\right) \right| &\leq& C\frac{
|{\psi_c}_j|}{ \rn}
\end{eqnarray*}
So that
\begin{eqnarray*}
\left| \log \left[ \frac{\pi_k(\theta-B_n)}{\pi_k(\theta)}\right]
\right| &\leq& C\frac{ \sum_{j=1}^k j^{\beta}|{\psi_c}_j|}{\rn}
\\
 &\leq & C\frac{ k^{\beta}}{ \rn }=o(1),
\end{eqnarray*}
for all $\gamma >1/2$, $1>\beta >1/2$ in the cases (D) and (PH), and condition (\ref{cond:prior}) is satisfied.

\item Student:
In the Student case for $g$ we can use the calculations made in the
Gaussian case since
\begin{eqnarray*}
& &  \sum_{j =1}^{k}\log \left( 1 + C j^{2\beta} \theta_j^2
\right)-\log \left( 1 + C j^{2\beta} (\theta_j -
t{\psi_c}_j/\rn)^2 \right) \\
 &=& 0\left( \sum_{j=1}^k
 j^{2\beta}[(\theta_j-t{\psi_c}_j/\rn)^2 - \theta_j^2] \right)
 \\
  &=& o(1)
\end{eqnarray*}
Therefore in all these cases condition (\ref{cond:prior}) is
satisfied.
\end{itemize}
Interestingly Theorem \ref{th:expo:BVM} shows that parametric sieve models (increasing sequence of models) have a behaviour which is a mix between parametric and nonparametric models. Indeed if the posterior distribution puts most of its mass on $k$'s large enough the posterior distribution has a Bernstein Von Mises property centered on the empirical (nonparametric MLE) estimator with the correct variance whereas if it allows for $k$'s that are not large enough (corresponding to $\sum_{j=1}^k
 j^{2\beta}[(\theta_j-t{\psi_c}_j/\rn)^2 - \theta_j^2]$  or $\Delta_\psi$ not small enough) then the posterior distribution is not asymptotically Gaussian with the right centering, nor with the right variance. An extreme case corresponds to the situation where $F_0(\Delta_\psi^2) \neq o(1)$ under the posterior distribution, which is equivalent to $$\exists k_0, \quad \mbox{s.t. }\forall \epsilon >0 \quad \mbox{liminf}_{n \rightarrow \infty} \Pp_0^n\left[ \Pp^\pi\left[ k_0|X^n\right] >\epsilon  \right]>0.$$
 For each $k>0$ fixed, if $\inf_{\theta\in \R^{k}} K(f_0,f_\theta)>0$, since the model is regular, there exists $c>0$ such that 
 $\Pp_0^n \left[ \Pp^\pi\left[ k|X^n\right] > e^{-nc }\right] \rightarrow 1$. Therefore, $F_0(\Delta_\psi^2) \neq o(1)$ under the posterior distribution if there exists $k_0$ such that 
  $\inf_{\theta\in \R^{k_0}} K(f_0,f_\theta)>0$, i.e. if there exists $\theta_0\in \R^{k_0}$ such that $f_0 = f_{\theta_0}$. In that case it can be proved that $\Pp^\pi[k_0|X^n]= 1+o_P(1)$, see \cite{CR}, and the Bernstein Von Mises theorem to be expected is the parametric one, under the model $\Theta_{k_0}$ which is regular. However, even if
  $\Delta_\psi =o_P(1)$, the posterior distribution might not satisfy the non parametric Bernstein Von Mises property with the correct centering. 
   We illustrate in the following section this issue in the special case of the cumulative distribution function.
 %%%%%%%%%%%%%%%%%%%%%%%%%%%%%%%%%%%%%%%%%%%%%%%%%%
\subsection{An example: the cumulative distribution function} \label{subsec:ex:cdf}
As a special case, consider the functional on $f$ to be the
cumulative distribution function calculated at a given point $x_0$. As
seen in Section \ref{sec:BVM}, ${\psi_c}(x) = \1_{x \leq x_0} -
F_0(x_0)$. We have $F_n(x_0) = P_n(\psi)$ and recall that the
variance of $\G_n(\psi)$ under $P_0$ is equal to $V_0 =
F_0(x_0)(1-F_0(x_0))$.

As an illustration, consider  the case of the Fourier basis. The case of wavelet bases is dealt with in the same way. In other words for $\la\geq 1$, $\phi_{2\la-1}(x)=\sqrt{2}\sin(2\pi\la x),$   $\phi_{2\la}(x)=\sqrt{2}\cos(2\pi\la x)$ and $\phi_0(x)= 1$.

\begin{corollary}
If the prior density $g$ on the coefficients is Gaussian or Laplace
then if $f_0 \in \mathcal S_\gamma$, with $\gamma \geq \beta$
 and if the prior on $k$ is the Dirac mass on $k_n^*$
then the posterior distribution of $\rn(F(x_0) - F_n(x_0))$ is
asymptotically Gaussian with mean 0 and variance $V_0$.

If the prior density $g$ is Student and if $\gamma\geq\beta >1$, 
 then the same result remains valid.
\end{corollary}

This result is a direct application of Theorem \ref{th:expo:BVM}.

\noindent

\textit{Counter-example:} In this remark we illustrate the fact that in the
case of a random $k$, which leads to an adaptive minimax rate of
convergence for the posterior distribution we might not have a
Bernstein - Von Mises theorem. Consider a density $f_0$ in the form

$$f_0 =\exp\left( \sum_{j \geq k_0} \theta_{0j} \phi_j(u) du
-c(\theta_0) \right)$$ where $k_0$ is fixed but can be large and
$\theta_{0,2j} =0$ and $$\theta_{0,2j-1} =  \sin(2\pi
jx)/[j^{\gamma+1/2}\sqrt{\log j}\log \log j].$$ Then for $J_1 >3$
\begin{eqnarray*}
\sum_{j\geq J_1} \theta_{0j}^2 j^{2\gamma}&\leq& \sum_{j \geq
 J_1}\frac{1}{ j \log j \log \log j^2} \\
  &\leq& \int_{j_1}^{\infty}\frac{ 1 }{ x\log x (\log \log x)^2} dx \\
   &=& \frac{ 1 }{ \log \log J_1 },
\end{eqnarray*}
and similarly
\begin{eqnarray} \label{LB:ex:1}
\sum_{j\geq J_1} \theta_{0j}^2 &\leq& \sum_{j \geq
 J_1}\frac{1}{ j^{2\gamma+1} \log j \log \log j^2} \nonumber \\
  &\leq& \int_{j_1}^{\infty}\frac{ 1 }{ x^{2\gamma+1}\log x (\log \log x)^2} dx \nonumber \\
   &=& \left[-\frac{ 1}{2\gamma x^{2\gamma} \log x (\log \log x)^2
   }\right]_{J_1}^\infty(1+o(1)) \nonumber\\
   &=&
 \frac{ 1}{2\gamma J_1^{2\gamma} \log J_1 (\log \log J_1)^2
   }(1+o(1))
\end{eqnarray}
when $J_1 \rightarrow \infty$.

Consider a Poisson distribution on $k$ with
parameter $\nu >0$ fixed    then for such $f_0$, if $k_n=
n^{1/(2\gamma+1)}(\log n)^{-2/(2\gamma+1)}(\log \log
n)^{-2/(2\gamma+1)}$ and  $k_1$ is large enough
$$\Pp^\pi[k \leq k_1k_n|X^n] = 1 + o(1).$$

 We now study the mean terms $\mu_{n,k}$ and we show that if $k \leq k_1 k_n$, 
$\mu_{n,k} \neq o(1)$ nor  can $\Pp^\pi(k|X^n)$  be neglected.

First note that when $k\rightarrow \infty $ $G_n(\Delta_\psi)=o(1)$
\begin{eqnarray} \label{eq:munk:1}
\mu_{n,k} &=& \rn F_0 \left[({\psi_c} - \Pi_{f_0,k} {\psi_c})
( l_0 - \Pi_{f_0,k}l_0) \right] \nonumber \\
 &=&\rn F_0 \left[(\sum_{j=k+1}^\infty \psi_{cj}\phi_j)
( l_0 - \Pi_{f_0,k}l_0) \right] \nonumber  \\
 &=&
\rn \int  \left[({\psi_c} - \Pi_{f_0,k} {\psi_c}) ( l_0 -
\Pi_{f_0,k}l_0) \right] \nonumber \\
 & &  + \rn  \int (f_0-1)  \left[({\psi_c} -
\Pi_{f_0,k} {\psi_c}) ( l_0 - \Pi_{f_0,k}l_0) \right]
\end{eqnarray}
We first consider the first term of the right hand side of
(\ref{eq:munk:1}).
\begin{eqnarray*}
\mu_{n,k,1} &=& \rn \int  \left[(\sum_{j=k+1}^\infty \psi_{cj}\phi_j)( l_0 - \Pi_{f_0,k}l_0) \right] \\
  &=&
  \rn \sum_{j=k+1}^\infty \psi_{cj} \theta_{0j} \\
   &=& \rn \sum_{l \geq k/2} \frac{ \sin^2(2\pi xl) }{
   (2l+1)^{\gamma+3/2}\log (2l+1)^{1/2} \log \log (2l+1) }
\end{eqnarray*}
and if $x=1/4$ we have
\begin{eqnarray*}
\mu_{n,k,1} &=& \rn \sum_{j \geq k/4-1/2} \frac{ 1 }{
   (4j+3)^{\gamma+3/2}(\log 4j+3)^{1/2} \log \log (4j+3)} \\
 &\leq& C \rn\frac{k^{-\gamma-1/2}}{\sqrt{\log k} \log \log k} \\
\mu_{n,k,1}  &\geq&
C'\rn \frac{k^{-\gamma-1/2}}{\sqrt{\log k} \log \log k}.
\end{eqnarray*}
Note that there exists $c>0$ such that for all $k \leq k_n$
$$\mu_{n,k,1} \geq c \sqrt{\log n}.$$

We now consider the second term of (\ref{eq:munk:1}). Let $M_{1,k}$
denote the projection on $(\phi_0,...,\phi_k)$ with respect to the
scalar product $<f,g>_2 = \int fg(u)du$ and note that
 $$
 \Pi_{f_0,k} l_0 = M_{1,k} l_0 + \Pi_{f_0,k}[\sum_{j=k+1}^\infty
 \theta_{0j} \phi_j ]$$
\begin{eqnarray*}
|\mu_{n,k,2}| &=& \left| \rn  \int (f_0-1)
\left[(\sum_{j=k+1}^\infty
\psi_j \phi_j) ( l_0 - \Pi_{f_0,k}l_0) \right] \right|\\
 &=&
\left|  \rn  \int (f_0-1)  \left[(\sum_{j=k+1}^\infty \psi_{cj}\phi_j)
( l_0 - M_{1,k}l_0) \right] \right| \\
 & & \quad +  \left| \rn  \int (f_0-1)
\left[(\sum_{j=k+1}^\infty
\psi_j \phi_j) ( M_{1,k}l_0 - \Pi_{f_0,k}l_0) \right] \right|\\
 &\leq & 2|f_0-1|_{\infty}\left(\sum_{j=k+1}^\infty
 \psi_{cj}^2\right)^{1/2}\left(\sum_{j=k+1}^\infty \theta_{0,j}^2
 \right)^{1/2} \\
  &\leq& C\rn|f_0-1|_{\infty}\frac{k^{-\gamma-1/2}}{\sqrt{\log k} \log \log k}
\end{eqnarray*}
By choosing $k_0$ large enough $|f_0-1|_\infty$ can be made as small
as need be so that we finally obtain that there exists $c>0$ such that for all $k\leq k_n$
$$ \mu_{n,k} \geq c\sqrt{\log n}.$$
Note that in case (D) with $\gamma<\beta$, the same calculations lead to 
$$ \mu_{n,k_n^*} \geq cn^{\frac{\beta-\gamma}{4\beta+2}}(\log n)^{-\frac{1}{2}}(\log\log n)^{-1}.$$
Thus in this case the posterior distribution is not asymptotically Gaussian with mean $F_n(x) $ and variance $F_0(x)(1-F_0(x))/n$. Whether it is asymptotically equivalent to a mixture of Gaussians is not clear. It would be a consequence of the way the posterior distribution of $k$ concentrates as $n$ goes to infinity. In the case (D), the posterior distribution is asymptotically Gaussian with mean $F_n(x)-\mu_{n,k_n^*}$.
%%%%%%%%%%%%%%%%%%%%%%%%%%%%%%%%%%%%%%%%%%%%%%%%%%%%%%%%%%
%%%%%%%%%%%%%%%%%%%%%%%%%%%%%%%%%%%%%%%%%%%%%%%%%%%%%%%%%%
\section{ Proofs} \label{sec:proofs}
In this section we prove Theorems \ref{th:sieve:1}, \ref{th:post:rate}  and \ref{th:expo:BVM}.
In the sequel, $C$ denotes a generic  positive constant whose value is of no importance.
%%%%%%%%%%%%%%%%%%%%%%%%%%%%%%%%%%%%%%%%%%%%%%%%%%%%%%%%%%%%%%%%%%%%%%%%%%
\subsection{Proof of Theorem \ref{th:sieve:1}} \label{sec:th:sieve:1}
Let $Z_n=\rn(\Psi(f) - \Psi(P_n)).$ We have
\begin{equation}\label{convAn}
\Pp^\pi\left\{ A_{n}|X^n\right\}=1 + o_{\Pp_0}(1).
 \end{equation}
So, it is enough to prove that conditionally on $A_n$ and $X^n$, the distribution of $Z_n$ converges to the distribution of a Gaussian variable whose variance is $F_0(\psi_c^2).$
This will be established if for any $t\in\R$,
\begin{equation}\label{convL}
\lim_{n\to +\infty} L_n(t)=\exp\left(\frac{t^2}{2}F_0\left[\psi_c^2\right]\right),
\end{equation}
 where $L_n(t)$ is the Laplace transform of $Z_n$ conditionally on $A_n$ and $X^n$:
\begin{eqnarray*}
 L_n(t)&=&\esp^\pi\left[\exp(t \rn(\Psi(f) - \Psi(P_n)) ) | A_n, X^n \right] \\
&=&\frac{\esp^\pi\left[\exp(t \rn(\Psi(f) - \Psi(P_n)) ) \1_{A_n}(f) | X^n \right]}
{\Pp^\pi\left\{ A_{n}|X^n\right\}}\\
&=&\frac{\int_{A_{n}} \exp\left(t\rn(\Psi(f) - \Psi(P_n))+l_n(f)-l_n(f_0)\right) d\pi(f) }{\int_{A_{n}} \exp\left(l_n(f)-l_n(f_0)\right) d\pi(f) }.
\end{eqnarray*}
We set for any $x$,
 \begin{eqnarray}\label{def:Bhn}
 B_{h,n}(x) = \int_0^1(1-u)e^{uh(x)/\sqrt{n}}du.
 \end{eqnarray}

so,
$$\exp\left(\frac{h(x)}{\sqrt{n}}\right)=1+\frac{h(x)}{\sqrt{n}}+\frac{h^2(x)}{n}B_{h,n}(x),$$
which implies that
$$f(x)-f_0(x)=f_0(x)\left(\frac{h(x)}{\sqrt{n}}+\frac{h^2(x)}{n}B_{h,n}(x)\right)$$
and
\begin{eqnarray*}
t\rn(\Psi(f) - \Psi(P_n))&=&-tG_n(\psi_c)+t\rn\left(\int\psi_c(x)(f(x)-f_0(x))dx\right)\\
&=&-tG_n(\psi_c)+tF_0(h\psi_c)+\frac{t}{\rn}F_0(h^2B_{h,n}\psi_c).
\end{eqnarray*}
Since $$l_n(f)-l_n(f_0)=-\frac{F_0(h^2)}{2}+ G_n(h) +  R_n(h),$$
we have
\begin{eqnarray*}
L_n(t)&=&\frac{\int_{A_{n}} \exp\left(G_n(h-t\psi_c)+tF_0(h\psi_c)+\frac{t}{\rn}F_0(h^2B_{h,n}\psi_c)-\frac{F_0(h^2)}{2}+R_n(h)\right)d\pi(f)}{\int_{A_{n}}\exp\left(-\frac{F_0(h^2)}{2}+ G_n(h) +  R_n(h)\right)d\pi(f)}\\
 &=&\frac{\int_{A_{n}} \exp\left(-\frac{F_0((h-t\bar{\psi}_{t,n})^2)}{2}+ G_n(h-t\bar{\psi}_{t,n}) +
R_n(h-t\bar{\psi}_{t,n})+U_{n,h}\right)d\pi(f)}{\int_{A_{n}}\exp\left(-\frac{F_0(h^2)}{2}+ G_n(h) +  R_n(h)\right)d\pi(f)},
\end{eqnarray*}
where straightforward computations show that
\begin{eqnarray*}
 U_{n,h}&=&tF_0(h(\psi_c-\bar{\psi}_{t,n}))+\frac{t^2}{2}F_0(\bar{\psi}_{t,n}^2)+R_n(h)-R_n(h-t\bar{\psi}_{t,n})+\frac{t}{\rn}F_0(h^2B_{h,n}\psi_c)\\
&=&tF_0(h\psi_c)+t\rn F_0(\bar{\psi}_{t,n})+\frac{t}{\rn}F_0(h^2B_{h,n}\psi_c)\\
&=&tF_0(h\psi_c)+n\log\left(
F_0\left[\exp\left(\frac{h}{\rn} - \frac{t\psi_c}{\rn}\right)\right]\right)+\frac{t}{\rn}F_0\left(h^2B_{h,n}\psi_c\right).
\end{eqnarray*}
Now, let us study each term of the last expression.
We have
\begin{eqnarray*}
F_0\left[\exp\left(\frac{h}{\rn} - \frac{t\psi_c}{\rn}\right)\right]&=&F_0\left[e^{\frac{h}{\rn}}\left(1-\frac{t\psi_c}{\rn}+\frac{t^2}{2n}\psi_c^2\right)\right]+0(n^{-\frac{3}{2}})\\
&=&1-\frac{t}{\rn}F_0\left[e^{\frac{h}{\rn}}\psi_c\right]+\frac{t^2}{2n}F_0\left[e^{\frac{h}{\rn}}\psi_c^2\right]+0(n^{-\frac{3}{2}}).
\end{eqnarray*}
So,
$$F_0\left[e^{\frac{h}{\rn}}\psi_c\right]=\frac{F_0[h\psi_c]}{\rn}+\frac{F_0[h^2B_{h,n}\psi_c]}{n}; \quad F_0\left[e^{\frac{h}{\rn}}\psi_c^2\right]
=F_0\left[\psi_c^2\right]+\frac{F_0[h\psi_c^2]}{\rn}+\frac{F_0[h^2B_{h,n}\psi_c^2]}{n}.$$
Note that, on $A_n$, we have $F_0(h^2)=0(nu_n^2)$ and $F_0\left(h^2 B_{h,n}\right)=o(n)$. Therefore, uniformly on $A_n$,
 \begin{eqnarray*}
F_0\left[\exp\left(\frac{h}{\rn} - \frac{t\psi_c}{\rn}\right)\right]&=&
1-\frac{t}{\rn}\left(\frac{F_0[h\psi_c]}{\rn}+\frac{F_0[h^2B_{h,n}\psi_c]}{n}\right)\\&&
\hspace{1cm}+\frac{t^2}{2n}\left(F_0\left[\psi_c^2\right]+\frac{F_0[h\psi_c^2]}{\rn}+\frac{F_0[h^2B_{h,n}\psi_c^2]}{n}\right)+o\left(n^{-1}\right)\\
&=&1 - \frac{t}{n}\left[ F_0[h\psi_c] +\frac{ F_0[h^2B_{h,n}\psi_c] }{\rn} - \frac{tF_0(\psi_c^2)}{2} + o(1) \right]\\&=&1+ o\left(n^{-1/2}\right)
\end{eqnarray*}
  and
\begin{eqnarray*}
n\log\left(F_0\left[\exp\left(\frac{h}{\rn} - \frac{t\psi_c}{\rn}\right)\right]\right)
&=&-t\left[ F_0(h\psi_c) +\frac{ F_0[h^2B_{h,n}\psi_c] }{\rn} - \frac{tF_0(\psi_c^2)}{2}  \right] +o(1).
\end{eqnarray*}
Finally,
$$U_{n,h}=\frac{t^2}{2}F_0\left[\psi_c^2\right]+o(1)$$
and  up to a multiplicative factor equal to $1 + o(1)$,
$$
L_n(t)=\exp\left(\frac{t^2}{2}F_0\left[\psi_c^2\right]\right)\frac{\int_{A_{n}} \exp\left(-\frac{F_0((h-t\bar{\psi}_{t,n})^2)}{2}+ G_n(h-t\bar{\psi}_{t,n}) +
R_n(h-t\bar{\psi}_{t,n})\right)d\pi(f)}{\int_{A_{n}}\exp\left(-\frac{F_0(h^2)}{2}+ G_n(h) +  R_n(h)\right)d\pi(f)}.
$$
Finally (A3) implies (\ref{convL}) and the theorem is proved.

%%%%%%%%%%%%%%%%%%%%%%%%%%%%%%%%%%%%%%%%%%%%%%%%%%%%%%%%%%%%%%%%%%%%%%%%%%%
\subsection{Proof of Theorem \ref{th:post:rate}} \label{sec:th:post:rate}
We first give a preliminary lemma which will be used extensively in the sequel.
\subsubsection{Preliminary lemma} \label{subsec:lemma}
Let us first state the following lemma.
\begin{lemma}\label{ineg}
Set $K_n=\{1,2,\dots, k_n\}$ with $k_n\in\Ne^*$.
Assume either of the following two cases:
\begin{itemize}
\item[-] $\ga>0$, $p=q=2$ when $\varPhi$ is the Fourier basis
\item[-] $0<\ga<r$, $2\leq p\leq \infty,$ $1\leq q \leq \infty$ when $\varPhi$ is the wavelet basis with $r$ vanishing moments.
\end{itemize}
Then the following results hold.
\begin{itemize}
\item[-] There exists a constant $c_{1,\varPhi}$ depending only on $\varPhi$  such that for any $\th=(\th_\la)_{\la}\in\R^{k_n}$,
\begin{equation}\label{normsup}
\left\|\sum_{\la\in K_n}\thl\phi_\la\right\|_\infty\leq c_{1,\varPhi}\sqrt{k_n}\norm{\th}_{\ell_2}.
\end{equation}
\item[-] If $\log(f_0)\in{\mathcal B}_{p,q}^{\ga}(R)$, then there exists  $c_{2,\ga}$ depending  on $\ga$ only such that
\begin{equation}\label{sob}
\sum_{\la\notin K_n}\thol^2\leq c_{2,\ga}\;R^2k_n^{-2\ga}.
\end{equation}
\item[-] If $\log(f_0)\in{\mathcal B}_{p,q}^{\ga}(R)$ with $\ga>\frac{1}{2}$, then there exists $c_{3,\varPhi,\ga}$ depending  on $\varPhi$ and $\ga$ only such that:
\begin{equation}\label{normsupsob}
\left\|\sum_{\la\notin K_n}\thol\phi_\la\right\|_{\infty}\leq c_{3,\varPhi,\ga}\;R \;k_n^{\frac{1}{2}-\ga}.
\end{equation}
\end{itemize}
\end{lemma}
\begin{proof}
Let us first consider the Fourier basis. We have:
\begin{eqnarray*}
\left\|\sum_{\la\in K_n}\thl\phi_\la\right\|_\infty&\leq&\sum_{\la\in K_n}|\thl|\times\norm{\phi_\la}_{\infty}\\
&\leq&|\!|\phi |\!|_\infty \sum_{\la\in K_n}|\thl|,
\end{eqnarray*}
which proves (\ref{normsup}). Inequality (\ref{sob}) follows from the definition of ${\mathcal B}_{2,2}^{\ga}=W^\ga$. To prove (\ref{normsupsob}), we use the following inequality: for any $x$,
\begin{eqnarray*}
 \left|\sum_{\la\notin K_n}\thol\phi_\la(x)\right|&\leq&|\!|\phi |\!|_\infty \sum_{\la\notin K_n}|\thol|\\
&\leq&|\!|\phi |\!|_\infty\left(\sum_{\la\notin K_n}|\la|^{2\ga}\thol^2\right)^{\frac{1}{2}}\left(\sum_{\la\notin K_n}|\la|^{-2\ga}\right)^{\frac{1}{2}}.
\end{eqnarray*}
Now, we consider the wavelet basis. Without loss of generality, we assume that $\log_2(k_n+1)\in\Ne^*$. We have for any $x$,
\begin{eqnarray*}
\left|\sum_{\la\in K_n}\thl\phi_\la(x)\right|&\leq&\left(\sum_{\la\in K_n}\thl^2\right)^{\frac{1}{2}}\left(\sum_{\la\in K_n}\phi_\la^2(x)\right)^{\frac{1}{2}}\\
&\leq&\norm{\th}_{\ell_2}\left(\sum_{0\leq j\leq \log_2(k_n)}\sum_{k=0}^{2^j-1}\varPsi_{jk}^2(x)\right)^{\frac{1}{2}}.
\end{eqnarray*}
Since $\varPsi(x)=0$ for $x\notin [-A,A]$,
$$\mbox{card}\left\{k\in\{0,\dots,2^j-1\}:\quad\varPsi_{jk}(x)\not=0\right\}\leq 3(2A+1).$$
(see \cite{Mal}, p. 282 or \cite{Mey}, p. 112). So, there exists $c_\varPsi$ depending only on $\varPsi$ such that
\begin{eqnarray*}
\left|\sum_{\la\in K_n}\thl\phi_\la(x)\right|&\leq&\norm{\th}_{\ell_2}\left(\sum_{0\leq j\leq \log_2(k_n)}3(2A+1)2^jc_\varPsi^2\right)^{\frac{1}{2}},
\end{eqnarray*}
which proves (\ref{normsup}). For the second point, we just use the inclusion ${\mathcal B}^{\ga}_{p,q}(R)\subset{\mathcal B}^{\ga}_{2,\infty}(R)$ and
$$
\sum_{\la\notin K_n}\thol^2=\sum_{j>\log_2(k_n)}\sum_{k=0}^{2^j-1}{\theta^2_0}_{jk}
\leq R^2\sum_{j>\log_2(k_n)}2^{-2j\ga}\leq \frac{R^2}{1-2^{-2\ga}}k_n^{-2\ga}.$$
Finally, for the last point, we have for any $x$:
\begin{eqnarray*}
\left|\sum_{\la\notin K_n}\thol\phi_\la(x)\right|&\leq&\sum_{j>\log_2(k_n)}\left(\sum_{k=0}^{2^j-1}{\theta^2_0}_{jk}\right)^{\frac{1}{2}}\left(\sum_{k=0}^{2^j-1}\varPsi^2_{jk}(x)\right)^{\frac{1}{2}}\\
&\leq& Ck_n^{\frac{1}{2}-\ga},
\end{eqnarray*}
where $C \leq R(3(2A+1))^{\frac{1}{2}}c_\varPsi(1-2^{\frac{1}{2}-\ga})^{-1}$.
\end{proof}

%%%%%%%%%%%%%%

\subsubsection{ Proof of Theorem \ref{th:post:rate}}
Denote for any $n$,
$$B_n(\epsilon_n) = \{f\in{\mathcal F}:\quad K(f_0,f) \leq
\epsilon_n^2, \ V(f_0,f) \leq \epsilon_n^2 \},$$
To prove Theorem \ref{th:post:rate}, we use the following version of the theorem on posterior
convergence rates. Its proof is not given, but it is a slight modification of Theorem 2.4 of \cite{GGVdV}.

%%%%%%%%%%%
\begin{thm}\label{th:rate:app}
Let $f_0$ be the true density. We assume that there exists a constant $c$ such that for any $n$, there exists ${\mathcal F}_n^*\subset {\mathcal F}$ and a prior $\pi$ on ${\mathcal F}$ satisfying the following conditions:
\begin{itemize}
\item[- (A)]$$\Pp^\pi\left\{{\mathcal{ F}^*_n}^c\right\} =o( e^{-(c+2)n\epsilon_n^2}).$$
\item[- (B)] For any $j\in\Ne^*$, let $$S_{n,j} = \{ f \in \mathcal F_n^*:\quad j \epsilon_n < h(f_0,f) \leq
(j+1) \epsilon_n\},$$
and $H_{n,j}$ the
Hellinger  metric entropy  of  $S_{n,j}$. There exists $J_{0,n}$ (that may depend on $n$) such that for all $j \geq J_{0,n},$
 $$H_{n,j} \leq (K-1)nj^2\epsilon_n^2,$$
where $K$ is an absolute constant.
\item[- (C)] Let $$B_n(\epsilon_n) = \{f\in{\mathcal F}:\quad K(f_0,f) \leq
\epsilon_n^2, \ V(f_0,f) \leq \epsilon_n^2 \}.$$
Then,
$$\Pp^\pi\left\{ B_n(\epsilon_n) \right\} \geq e^{-cn
\epsilon_n^2}.$$
\end{itemize}
We have:
 $$\Pp^\pi\left\{ f:\quad h(f_0,f) \leq J_{0,n} \epsilon_n |X^n\right\}
 =1+o_P(1)$$
% and for all $A_n \subset \mathcal F$ such that $\pi\left[ A_n \right] =o( e^{-(c+2)n\epsilon_n^2})$ we have$$P^\pi\left[ A_n |X^n\right]=o_P(1)$$
\end{thm}
%%%%%%%%%%%%%%%%%%%
To prove Theorem \ref{th:post:rate} it is thus enough to prove that  conditions (A), (B) and (C) of the previous result are satisfied.
We consider $(\Lambda_n)_n$ the increasing sequence of subsets of $\Ne^*$ defined by
$\Lambda_n=\{1,2,\dots, l_n\}$ with $l_n\in\Ne^*$. For any $n$, we set:
\begin{eqnarray*}
{\mathcal     F}_n^* =\left\{f_\th\in{\mathcal     F}_{l_n}:\quad     f_\theta    =
 \exp\left(\sum_{\la\in \Lambda_n}\th_\la \phi_\la -c(\th) \right), \ \norm{\th}_{\ell_2}\leq w_n\right\},
 \end{eqnarray*}
with
$$w_n=\exp(w_0n^\rho(\log n)^q), \quad \rho >0$$
Recall that
\begin{itemize}
\item[-] $\epsilon_n= \epsilon_0 n^{-\frac{\ga}{2\ga+1}}(\log n)^{\frac{\ga}{2\ga+1}} $ in case (PH)
\item[-] $\epsilon_n= \epsilon_0 n^{-\frac{\be}{2\be+1}} $ in case (D).
\end{itemize}
Define $l_n$ by
\begin{equation}\label{ln}
l_n=\frac{ l_0n\ep_n^2}{ L(n) },
\end{equation}
where $l_0$ is some positive constant.
% So,
%\begin{itemize}
%\item[-]
%$$l_n\sim \ep_0^2n^{\frac{1}{2\ga+1}}(\log n)^{\frac{2\ga}{2\ga+1}}$$ in case (PH),
%\item[-] $l_n\sim \ep_0^2n^{\frac{1}{2\beta+1}}\log n$ in case (D).
%\end{itemize}
When $\ga,\be>\frac{1}{2}$, we have
\begin{equation}\label{le}
l_n\ep_n^2\to 0.
\end{equation}
%%%%%%%%%%%%%%%
\paragraph{Proof of condition (A):} We have, since $\sum_k \tau_k < \infty$
 \begin{eqnarray*}
 \pi\left\{ {\mathcal F_n^*}^c \right\} &\leq& \sum_{k>l_n}p(k)+\Pp^\pi\left\{\sum_{k\leq  l_n}\theta_k^2>w_n^2\right\}\\
&\leq&
C\exp\left(-l_n L(l_n)\right)+\sum_{k\leq l_n}\Pp^\pi\left\{\frac{\theta_k^2}{ \tau_k } >  w_n^2\right\}\\
&\leq&C\exp\left(-l_0n\ep_n^2\right)+\sum_{k\leq l_n}\Pp^\pi\left\{\exp\left(\frac{|\theta_k|^p}{2\tau_k^{p/2}}\right)>\exp\left(\frac{w_n^p}{2}\right)\right\}\\
%&\leq& C\exp\left(-l_0n\ep_n^2\right)+ \exp\left(-\frac{w_n^t}{2}\right)\sum_{k\leq
%k_n}\E_{\pi}\left[\exp\left(\frac{|\theta_k|^t}{2}\right)\right]\\
&\leq&C\exp\left(-l_0n\ep_n^2\right)+ Cl_n\exp\left(-\frac{w_n^p}{2}\right)\\
&\leq&C\exp\left(-l_0n\ep_n^2\right)+ C\exp\left(-n^H\right)
\end{eqnarray*}
 for any positive $H>0$.
 Hence, $$\pi\left\{{\mathcal F^*_n}^c\right\}\leq C\exp\left(-(l_0-1)n\ep_n^2\right)$$
and Condition (A) is proved.
%%%%%%%%%%%%%%%
\paragraph{Proof of condition (B):}
We  apply Lemma \ref{ineg} with $K_n=\Lambda_n$ and $k_n=l_n$. For  this purpose, we show  that the Hellinger  distance between two
functions  of ${\mathcal  F^*_n}$ is  related  to the  $\ell_2$-distance of  the
associated coefficients.  So, let us consider $f_{\theta}$ and $f_{\theta'}$ belonging to
${\mathcal F^*_n}$ with
$$f_{\theta}=\exp\left(\sum_{\la\in \Lambda_n}\thl\phi_\la-c(\th)\right),\quad
f_{\theta'}=\exp\left(\sum_{\la\in \Lambda_n}\thl'\phi_\la-c(\th')\right).$$
Let  us  assume  that  $\norm{\th'-\th}_{\ell_1}\leq  c_1\ep_nl_n^{-1/2} $  with  $c_1$  a
positive constant,  then using (\ref{normsup}) and (\ref{le}),
$$\left\|\sum_{\la\in \Lambda_n}(\thl'-\thl)\phi_\la\right\|_\infty\leq
C\sqrt{l_n}\norm{\th'-\th}_{\ell_2}\leq C\sqrt{l_n}\norm{\th'-\th}_{\ell_1}\leq
Cc_1\ep_n\to 0$$ and
\begin{eqnarray*}
\left|c(\te)-c(\te')\right|&=&\left|\log\left(\int_0^1
f_\th(x)\exp\left(\sum_{\la\in \Lambda_n}(\thl'-\thl)\phi_\la(x)\right)\right)\right|\\
&\leq&\left|\log\left(1+C\norm{\sum_{\la\in \Lambda_n}(\thl'-\thl)\phi_\la}_{\infty}\right)\right|\\
&\leq&C\norm{\sum_{\la\in \Lambda_n}(\thl'-\thl)\phi_\la}_{\infty}.
\end{eqnarray*}
Then,
\begin{eqnarray}\label{V1}
h^2(f_{\th},f_{\th'})&=&\int f_\theta(x)\left(\exp\left(\frac{1}{2}\sum_
{\la\in \Lambda_n}(\thl'-\thl)\phi_\la(x)+\frac{1}{2}\left(c(\th)-c(\th')\right)\right)-1\right)^2dx\nonumber\\
&\leq&\int_0^1  f_\th(x)\left(\exp\left(C\norm{\sum_{\la\in \Lambda_n}(\thl'-\thl)\phi_\la}_\infty\right)-1\right)^2dx\nonumber\\
&\leq&C\norm{\sum_
{\la\in \Lambda_n}(\thl-\thl')\phi_\la}_\infty^2\nonumber\\
&\leq& C l_n\norm{\th-\th'}_{\ell_1}^2\leq C l_n^2\norm{\th-\th'}_{\ell_2}^2
\end{eqnarray}
The next lemma establishes a converse inequality.
\begin{lemma}\label{h2}
There exists a constant $c\leq 1/2$ depending on $\ga$, $\beta$,$R$ and $\varPhi$ such that if
$$(j+1)^2\ep_n^2l_n\leq  c\times\min\left(c_0,(1-e^{-1})^2\right)$$
then for $f_\th\in S_{n,j}$,
$$\norm{\th_0-\th}_{\ell_2}^2\leq\frac{1}{c_0c}(\log n)^2h^2(f_0,f_\th).$$
\end{lemma}

\begin{proof}
Using Theorem 5 of \cite{WS}, with $M_1=\left(\int_0^1\frac{f_0^2(x)}{f_\th(x)}dx\right)^{\frac{1}{2}}$, if
$$h^2(f_0,f_\th)\leq \frac{1}{2}(1-e^{-1})^2,$$ we have
\begin{eqnarray}\label{argumentSW}
V(f_0,f_\th)&\leq&5h^2(f_0,f_\th)\left(|\log M_1|-\log(h(f_0,f_\th)\right)^2.
\end{eqnarray}
But
\begin{eqnarray*}
M_1&=&\int_0^1f_0(x)\exp\left(\sum_
{\la\in \Lambda_n}(\thol-\thl)\phi_\la(x)+\sum_
{\la\notin \Lambda_n}\thol\phi_\la(x)-c(\th_0)+c(\th)\right)dx\\
&\leq&\int_0^1f_0(x)\exp\left(C[\sqrt{l_n}\norm{\th_0-\th}_{\ell_2}+R \ell_n^{\frac{1}{2}-\ga}]-c(\th_0)+c(\th)\right)dx,
\end{eqnarray*}
by using (\ref{normsup}) and (\ref{normsupsob}). Furthermore,
\begin{eqnarray}\label{contc}
\left|c(\th_0)-c(\th)\right|
&\leq&C[\sqrt{l_n}\norm{\th_0-\th}_{\ell_2}+R \ l_n^{\frac{1}{2}-\ga}].
\end{eqnarray}
So,
$$|\log M_1|\leq C[\sqrt{l_n}\norm{\th_0-\th}_{\ell_2}+R \ l_n^{\frac{1}{2}-\ga}].$$
Finally, since $f_{\theta} \in S_{n,j}$ for $j\geq 1$,
\begin{eqnarray*}
V(f_0,f_\th)&\leq&5h^2(f_0,f_\th)\left(C[\sqrt{l_n}\norm{\th_0-\th}_{\ell_2}+R \ l_n^{\frac{1}{2}-\ga}]-\log(\ep_n)\right)^2\\
&\leq&Ch^2(f_0,f_\th)\left(l_n\norm{\th_0-\th}_{\ell_2}^2+(\log n)^2\right).
\end{eqnarray*}
 Since $f_0(x)\geq c_0$ for any $x$ and $\int_0^1\phi_\la(x)dx=0$ for any $\la\in\Lambda$, we have
 \begin{eqnarray}\label{V2}
 V(f_0,f_\theta)   &\geq&c_0 \norm{\th_0-\th}_{\ell_2}^2.
 \end{eqnarray}
Combining (\ref{V1}) and (\ref{V2}), we conclude that
\begin{eqnarray*}
 \norm{\th_0-\th}_{\ell_2}^2 %&\leq&Ch^2(f_0,f_\th)\left(l_n\norm{\th_0-\th}_{\ell_2}^2+(\log n)^2\right)\\
&\leq&C(\log n)^2h^2(f_0,f_\th),
\end{eqnarray*}
if $h^2(f_0,f_\th)l_n\leq (j+1)^2\ep_n^2l_n\leq 1/(2C).$
Lemma \ref{h2} is proved by taking $c=(\max(C,1))^{-1}/2$.
\end{proof}
Now, under assumptions of Lemma \ref{h2}, using (\ref{V1}), we obtain
$$H_{n,j}\leq\log\left(\left(Cl_n(j+1)\log n\right)^{l_n}\right)\leq l_n\log\left(C\ep_n^{-1}\sqrt{l_n}\log n\right).$$
Then,  since $l_nL(n)=l_0 n\ep_n^2$, we have $$H_{n,j} \leq (K-1)nj^2\epsilon_n^2$$ as soon as
$$J_{0,n}^2\geq \frac{j_0 \log n}{L(n)},$$
where $j_0$ is a constant and condition (B) is satisfied for such $j$'s.  Now, let $j$ be such that
\begin{equation}\label{nonhypl2}
c(j+1)^2\ep_n^2l_n> \min\left(\frac{c_0}{2},\frac{1}{2}(1-e^{-1})^2\right).
\end{equation}
In this case, since for $f_\th\in {\mathcal F}_n^*$,
$$\norm{\th}_{\ell_1}\leq \sqrt{l_n}\norm{\th}_{\ell_2}\leq\sqrt{l_n}w_n,$$ for $n$ large enough,
$$H_{n,j}\leq\log\left(\left(Cl_nw_n\ep_n^{-1}\right)^{l_n}\right)\leq 2l_n\log(w_n)\leq 2w_0l_nn^\rho(\log n)^q.$$
Then, using (\ref{nonhypl2}), condition (B) is satisfied if $w_0$ and $q$ are small enough and if
$$l_n^2(\log n)^q\leq n^{1-\rho},$$
which is true for $n$ large enough, since $\ga,\be>\frac{1}{2}$, for $\rho$ small enough.

%%%%%%%%%%%%%%%
\paragraph{Proof of condition (C)} %In the sequel, $D$ is a constant that only depends on $\norm{f_0}_\infty$, $\ga$, $R$ and $\Phi$ and that may change from line to line. Now, we use the set $K_n$ defined in Lemma \ref{ineg},
Let $k_n\in \N$, going to $\infty$ and $K_n=\{1,...,k_n\}$,
we assume that $\theta$ belongs to $A(u_n)$ where
\begin{equation}\label{Aun}
A(u_n)=\left\{\theta:\quad\thl=0 \mbox{ for every } \la\notin K_n\mbox{ and } \sum_{\la\in K_n}(\thol-\thl)^2\leq u_n^2\right\},
\end{equation}
where $u_n$ goes to 0 such that
\begin{equation}\label{ku1}
\sqrt{k_n}u_n\to 0.
\end{equation}
We  define for any $\la\in \La$,
$$\be_{\la}(f_0)=\int_0^1 \phi_{\la}(x)f_0(x)dx.$$
Let us introduce the following notations:
$$\foL=\exp\left(\sum_{\la\in K_n}\thol\phi_{\la}(x)-c(\thoL)\right),\quad \foLL=\exp\left(\sum_{\la\notin K_n}\thol\phi_{\la}(x)-c(\thoLL)\right).$$
%where
%$$c(\thoL)=\log\left(\int_0^1\exp\left(\sum_{\la\in
%K_n}\thol\phi_{\la}(x)\right)dx\right),\quad c(\thoLL)=\log\left(\int_0^1\exp\left(\sum_{\la\notin K_n}\thol\phi_{\la}(x)\right)dx\right). $$
We have
\begin{eqnarray*}
K(f_0,\foL) %&=&\int_0^1\log\left(\frac{f_0(x)}{\foL(x)}\right)f_0(x)dx\\
%&=&\int_0^1\log\left[\exp\left(\sum_{\la\notin K_n}\thol\phi_\la(x)+c(\thoL)- c(\th_0)\right)\right]f_0(x)dx \\
&=& \sum_{\la\notin K_n}\thol \be_\la(f_0)  + c(\thoL)- c(\th_0)\\
 &=&\sum_{\la\notin K_n}\thol \be_\la(f_0) + \log\left(\int_0^1 f_0(x) e^{-\sum_{\la\notin K_n} \thol\phi_\la(x)}dx\right) .
\end{eqnarray*}
Using inequality (\ref{normsupsob}) of Lemma \ref{ineg} and a Taylor expansion of the function $e^x$ we obtain
\begin{eqnarray*}
&&\int_0^1 f_0(x) e^{-\sum_{\la\notin K_n} \thol\phi_\la(x)}dx\\
%&=&\int_0^1 f_0(x)\left(1-\sum_{\la\notin K_n} \thol\phi_\la(x)
%+\frac{1}{2}\left(\sum_{\la\notin K_n} \thol\phi_\la(x)\right)^2\times\left(1+o(1)\right)\right)dx\\
&=&1-\sum_{\la\notin K_n} \thol\be_\la(f_0)+\frac{1}{2}\int_0^1 f_0(x)\left(\sum_{\la\notin K_n} \thol\phi_\la(x)\right)^2dx\times\left(1+o(1)\right).
\end{eqnarray*}
We have
\begin{eqnarray*}
\left|\sum_{\la\notin K_n} \thol\be_\la(f_0)\right|
&\leq&\|f_0\|_2\left(\sum_{\la\notin K_n}\thol^2\right)^{\frac{1}{2}}
\end{eqnarray*}
and
\begin{eqnarray*}
\int_0^1 f_0(x)\left(\sum_{\la\notin K_n} \thol\phi_\la(x)\right)^2dx&\leq&\|f_0\|_{\infty}\sum_{\la\notin K_n}\thol^2
\end{eqnarray*}
So,
\begin{eqnarray*}
\log\left(\int_0^1 f_0(x) e^{-\sum_{\la\notin K_n} \thol\phi_\la(x)}dx\right)
&=&-\sum_{\la\notin K_n}\thol\be_\la(f_0)-\frac{1}{2}\left(\sum_{\la\notin K_n}\thol \be_\la(f_0)\right)^2\\&&+\frac{1}{2}\int_0^1
f_0(x)\left(\sum_{\la\notin K_n}\thol\phi_\la(x)\right)^2dx+o\left(\sum_{\la\notin K_n}\thol^2\right).\\
\end{eqnarray*}
So, finally,
\begin{eqnarray*}
K(f_0,\foL) &=&
 \frac{1}{2}\int_0^1
f_0(x)\left(\sum_{\la\notin K_n}\thol\phi_\la(x)\right)^2dx-\frac{1}{2}\left(\sum_{\la\notin K_n}\thol\be_\la(f_0)\right)^2
 +o\left(\sum_{\la\notin K_n}\thol^2\right)
 \end{eqnarray*}
This implies that for $n$ large enough,
$$
K(f_0,\foL) \leq \|f_0\|_{\infty}\sum_{\la\notin K_n}\thol^2\leq Dk_n^{-2\gamma}.$$
Now, if $f_\theta \in  \mathcal F_{k_n}$ with
$f_\theta=\exp\left(\sum_{\la\in K_n}\thl\phi_\la-c(\th)\right),$
we have
\begin{eqnarray*}
K(f_0,f_{\theta})
&=&K(f_0,\foL)+ \sum_{ \la\in K_n}(\thol-\thl) \be_\la(f_0)- c(\theta_{0K_n}) + c(\theta) \\
&\leq&Dk_n^{-2\ga}+
\sum_{ \la\in K_n}(\thol-\thl) \be_\la(f_0)- c(\theta_{0K_n}) + c(\theta).
\end{eqnarray*}
We set for any $x$,
$$T(x)=\sum_{\la\in K_n}(\thl-\thol)\phi_\la(x)
.$$
Using (\ref{normsup}),
$$\left\|T\right\|_{\infty}\leq C\sqrt{k_n}u_n\to
0.$$
So,
$$\int_0^1f_{0K_n}(x)\exp(T(x))dx=1+\int_0^1f_{0K_n}(x)T(x)dx+\int_0^1f_{0K_n}(x)T^2(x)v(n,x)dx,$$
where $v$ is a bounded function. Since $\log(1+u)\leq u$ for any $u>-1$,
 for $\theta\in A(u_n)$ and $n$ large enough,
\begin{eqnarray*}
|-c(\theta_{0K_n})+c(\theta)|&=&
\left|\log\left( \int_0^1 f_{0K_n}(x)e^{T(x)}dx\right)\right| \\
&\leq& \int_0^1f_{0K_n}(x)T(x)dx+\int_0^1f_{0K_n}(x)T^2(x)v(n,x)dx\\
&\leq&\sum_{\la\in K_n}(\thl-\thol)\be_\la(f_{0K_n})+Dk_nu_n^2.
\end{eqnarray*}
So,
\begin{eqnarray*}
K(f_0,f_{\theta})
&\leq&Dk_n^{-2\ga}+
\sum_{ \la\in K_n}(\thol-\thl) \left(\be_\la(f_0)- \be_\la(f_{0K_n})\right)\\
&\leq&Dk_n^{-2\ga}+ u_n\|f_0- f_{0K_n}\|_2
\end{eqnarray*}
%We have \begin{eqnarray*}|\be_\la(f_0)- \be_\la(f_{0K_n})|&=&\left|\int_0^1 \phi_\la(x)(f_0(x)- f_{0K_n}(x))dx\right|\\&\leq&\|f_0- f_{0K_n}\|_2\end{eqnarray*}
Using (\ref{normsupsob}), we have
\begin{eqnarray*}
\|f_0- f_{0K_n}\|_2^2 &\leq& \|f_0\|_\infty^2 \int_0^1\left(1- \exp\left(-\sum_{\la\notin K_n} \thol\phi_\la(x)-c(\thoL)+c(\th_0)\right)\right)^2dx.
\end{eqnarray*}
and
\begin{eqnarray*}
\left|c(\thoL)-c(\th_0)\right| %&=&\left|\log\left(\int_0^1f_0(x)\exp\left(-\sum_{\la\notin K_n} \thol\phi_\la(x)\right)dx\right)\right|\\
&\leq&\norm{\sum_{\la\notin K_n} \thol\phi_\la}_{\infty}.
\end{eqnarray*}
Finally,
$$
\|f_0- f_{0K_n}\|_2 \leq D\norm{\sum_{\la\notin K_n} \thol\phi_\la}_{\infty}\leq Dk_n^{\frac{1}{2}-\ga}.
$$
and
\begin{equation}\label{K}
K(f_0,f_{\theta})
\leq Dk_n^{-2\ga}+ Du_n k_n^{\frac{1}{2}-\ga}.
\end{equation}
We now bound $V(f_0,f_\theta)$. For this purpose, we refine the control of $\left|c(\thoL)-c(\th_0)\right|$:
\begin{eqnarray*}
\left|c(\thoL)-c(\th_0)\right|&=&\left|\log\left(\int_0^1f_0(x)\exp\left(-\sum_{\la\notin K_n} \thol\phi_\la(x)\right)dx\right)\right|\\
&=&\left|\log\int_0^1 f_0(x)\left(1-\sum_{\la\notin K_n}\thol\phi_\la(x)+w(n,x)\left(\sum_{\la\notin K_n}\thol\phi_\la(x)\right)^2\right)dx\right|,
\end{eqnarray*}
where $w$ is a bounded function. So,
\begin{eqnarray*}
\left|c(\thoL)-c(\th_0)\right|
&\leq&D\left(\sum_{\la\notin K_n}\left|\thol\be_\la(f_0)\right|+\int_0^1\left(\sum_{\la\notin K_n}\thol\phi_\la(x)\right)^2dx\right)\\
&\leq&D\left(\sum_{\la\notin K_n}\thol^2\right)^{\frac{1}{2}}\leq Dk_n^{-\ga}.
\end{eqnarray*}
In addition,
\begin{eqnarray*}
\left|c(\theta_{0K_n})-c(\theta)\right|&\leq&\sum_{\la\in K_n}\left|\thl-\thol\right|\left|\be_\la(f_{0K_n})\right|+Dk_nu_n^2\\
&\leq&u_n\left(\norm{f_0-f_{0K_n}}_2+\norm{f_0}_2\right)+Dk_nu_n^2\\
&\leq&Du_n+Dk_nu_n^2
\end{eqnarray*}
Finally,
\begin{eqnarray}\label{V}
V(f_0,f_\theta)  &\leq&u_n^2+Dk_n^{-2\ga}+Dk_nu_n^2.
\end{eqnarray}
Now, let us consider the case (PH). We take $k_n$ and $u_n$ such that
\begin{equation}\label{ku}
k_n^{-2\ga}\leq k_0\ep_n^2\quad\mbox{ and } u_n=u_0\ep_nk_n^{-\frac{1}{2}},
\end{equation}
 where $k_0$ and $u_0$ are constants depending on $\norm{f_0}_\infty$, $\ga$, $R$ and $\varPhi$. If $k_0$ and $u_0$ are small enough, then, by using (\ref{K}) and (\ref{V}),
\[K(f_0,f_\th)\leq \ep_n^2\quad \mbox{and}\quad V(f_0,f_\theta)\leq \ep_n^2.\]So,
Condition (C) is satisfied if
$$\Pp^\pi\left\{A(u_n)\right\}\geq e^{-cn\ep_n^2},$$
where, $A(u_n)$ is defined in (\ref{Aun}). We have:
\begin{eqnarray*}
\Pp^\pi\left\{A(u_n)\right\}&\geq&\Pp^\pi\left\{\theta:\quad\sum_{\la\in K_n}(\thl-\thol)^2\leq u_n^2\right\}\times \exp\left(-c_1k_nL(k_n)\right\}
\end{eqnarray*}
The prior on $\theta $ implies that
\begin{eqnarray*}
P_1&=&\Pp^\pi\left\{\theta:\quad\sum_{\la\in K_n}(\thl-\thol)^2\leq u_n^2\right\}\\
&\geq&\Pp^\pi\left\{\theta:\quad\sum_{\la\in K_n}\left|\sqrt{\tau_0}\la^{-\beta}G_\la-\thol\right|\leq u_n\right\}\\
&=&\Pp^\pi\left\{\theta:\quad\sum_{\la\in K_n}\la^{-\beta}\left|G_\la-\tau_0^{-\frac{1}{2}}\la^{\beta}\thol\right|\leq \tau_0^{-\frac{1}{2}}u_n\right\}\\
&=&\int...\int 1_{\left\{\sum_{\la\in K_n}\la^{-\beta}\left|x_\la-\tau_0^{-\frac{1}{2}}\la^{\beta}\thol\right|\leq \tau_0^{-\frac{1}{2}}u_n\right\}}\prod_{\la\in K_n}g(x_\la)dx_\la\\
&\geq&\int...\int 1_{\left\{\sum_{\la\in K_n}\la^{-\beta}\left|y_\la\right|\leq \tau_0^{-\frac{1}{2}}u_n\right\}}\prod_{\la\in K_n}g\left(y_\la+\tau_0^{-\frac{1}{2}}\la^{\beta}\thol\right)dy_\la.
\end{eqnarray*}
Using (\ref{sob}), when $\ga\geq \be$, we have
$\sup_{\la\in K_n}\left|\tau_0^{-\frac{1}{2}}\la^{\beta}\thol\right|<\infty$ and since
\begin{equation}\label{ku2}
\sup_n\left\{\tau_0^{-\frac{1}{2}}k_n^\beta u_n\right\}<\infty
\end{equation}
using assumptions on the prior, there exists a constant $D_3$ such that
\begin{eqnarray}\label{c1}
P_1&\geq& D_3^{k_n}\int...\int 1_{\left\{\sum_{\la\in K_n}\la^{-\beta}\left|y_\la\right|\leq \tau_0^{-\frac{1}{2}}u_n\right\}}\prod_{\la\in K_n}dy_\la\nonumber\\
&\geq&  D_3^{k_n}\int...\int 1_{\left\{\sum_{\la\in K_n}\left|y_\la\right|\leq \tau_0^{-\frac{1}{2}}u_n\right\}}\prod_{\la\in K_n}dy_\la\nonumber\\
&\geq&\exp\left(-D_4k_n\log n\right),
\end{eqnarray}
where $D_4$ is a constant.
When $\gamma <\beta$, since there exists $a,b>0$ such that $\forall |y| \leq M $  for some positive $M$
$$g(y+u) \geq a \exp( -b |u|^{p^*} ) $$
using the above calculations we obtain if $p^* \leq 2$
\begin{eqnarray*}
P_1&\geq& D_3^{k_n}\exp\{-C \sum_{\lambda\in K_n} \lambda^{p^*\beta} |\theta_{0\lambda}|^{p^*} \} \sum \int...\int 1_{\left\{\sum_{\la\in K_n}\la^{-\beta}\left|y_\la\right|\leq \tau_0^{-\frac{1}{2}}u_n\right\}}\prod_{\la\in K_n}  dy_\la \\
&\geq& \exp\left[ -Ck_n^{1-p^*/2+ \beta- \gamma} \right]    \exp\left(-D_4k_n\log n\right) \\
 &\geq & \exp\left(-(D_4+1)k_n\log n\right) \quad \mbox{ if } \beta \leq 1/2 +p^*/2
\end{eqnarray*}
and if $t>2$
\begin{eqnarray*}
P_1&\geq& D_3^{k_n}\exp\{-C \sum_{\lambda\in K_n} \lambda^{p^*\beta} |\theta_{0\lambda}|^{p^*} \}\exp\left(-D_4k_n\log n\right)\\
 &\geq & \exp\left(-(D_4+1)k_n\log n\right) \quad \mbox{ if } \beta \leq 1/2 +1/p^*
\end{eqnarray*}

So, Condition (C) is established as soon as $D_4k_n\log n\leq cn\ep_n^2$. Using (\ref{ku}), this can be satisfied if and only if we take $k_n$ such that
\begin{equation}\label{ku3}
k_0^{-\frac{1}{2\ga}}\ep_n^{-\frac{1}{\ga}}\leq k_n\leq \frac{cn\ep_n^2}{D_4\log n},
\end{equation}
which is possible if and only if $\ep_0$ is large enough. In particular, this implies that
$$\sup_{n}\left\{\ep_n\left(\frac{\log n}{n}\right)^{-\frac{\ga}{2\ga+1}}\right\}<\infty.$$ Note that when $k_n$ satisfies (\ref{ku3}), Conditions (\ref{ku1}) and (\ref{ku2}) are satisfied as well.

Similar computations show the result for the case (D).

%%%%%%%%%%%%%%%%%%%%%%%%%%%%%%%%%
\subsection{Proof of Theorem  \ref{th:expo:BVM}} \label{sec:th:expo:BVM}
Our goal is to prove conditions (A1), (A2) and (A3) of Section \ref{subsec:BVM:line} to apply Theorem \ref{th:sieve:1}.
Let $\epsilon_n$ be the posterior concentration rate as obtained in Theorem \ref{th:post:rate}.
%Recall that from Theorem \ref{th:post:rate} the posterior concentration rate $\epsilon_n$ is bounded by
% $n^{-\gamma/(2\gamma+1)}$ up to a $\log n$ term so that $\epsilon_n = o(n^{-1/4})$.

Let us consider $f=f_{\theta}\in{\mathcal F}_{k}$ for $1\leq k\leq l_n$, where $l_n = l_0 n\epsilon_n^2/L(n)$ in the case of type (PH) priors and $l_n = k_n^*$ in the case of type (D) priors.
First, using the same upper bound as in the proof of Lemma \ref{h2} we have
\begin{equation}\label{Vbis}
V(f_0,f)
 \leq 2C(\log n)^2 \epsilon_n^2,
\end{equation}
as soon as $h(f_0,f) \leq \epsilon_n$.
Thus, using (\ref{post:rate:1}), we have
$$\Pp^\pi\left\{ A_{u_n}^1|X^n\right\} = 1 + o_{\Pp_0}(1)$$
with $u_n=u_0 (\log n)^2 \epsilon_n^2 $,
for a constant $u_0$ large enough. Note that we can restrict ourselves to $A_{u_n}^1 \cap ( \cup_{k \leq l_n} \mathcal F_k )$, since
$\Pp^\pi\left[ ( \cup_{k \leq l_n} \mathcal F_k )^c \right] \leq e^{-cn\epsilon_n^2}$ for any $c>0$ by choosing  $l_0$ large enough, see the proof of Theorem \ref{th:post:rate}.

To establish (A2), we observe that
\begin{eqnarray*}
\norm{\log f_\theta-\log f_0}_{\infty}&\leq&\norm{\sum_{\la\in\Ne^*}(\thol-\thl)\phi_\la}_{\infty}+|c(\theta)-c(\theta_0)|\\
&\leq&C\left(\sqrt{l_n}\norm{\th-\th_0}_{\ell_2}+l_n^{\frac{1}{2}-\gamma}\right)=0(1),
\end{eqnarray*}
by using Lemma \ref{ineg} and (\ref{contc}). So, (A2) is implied by (A1). Now, let us establish (A3). Denote $A_n$ the set defined in assumption (A2) and restricted to $( \cup_{k \leq l_n} \mathcal F_k )$. For any $t$, we study the term
\begin{eqnarray*}
I_n&=&\frac{\int_{A_{n}}\exp\left(-\frac{F_0((h_f-t\bar{\psi}_{t,n})^2)}{2}+ G_n(h_f-t\bar{\psi}_{t,n}) +
R_n(h_f-t\bar{\psi}_{t,n})\right) d\pi(f) }{\int_{A_{n}}\exp\left(-
\frac{F_0(h_f^2)}{2}+ G_n(h_f) +  R_n(h_f)\right)d\pi(f) }\\
&=&\frac
{\sum_{1\leq k\leq l_n}p(k)\int_{A_{n}\cap {\mathcal F}_{k}}\exp\left(-\frac{F_0((h_f-t\bar{\psi}_{t,n})^2)}{2}+ G_n(h_f-t\bar{\psi}_{t,n}) +R_n(h_f-t\bar{\psi}_{t,n})\right) d\pi_k(f) }
{\sum_{1\leq k\leq l_n}p(k)\int_{A_n\cap {\mathcal F}_{k}}\exp\left(-\frac{F_0(h_f^2)}{2}+ G_n(h_f) +  R_n(h_f)\right)d\pi_k(f)}.
\end{eqnarray*}
If we set
$$b_{n,k,t} = \frac{t\Pi_{f_0,k}\psi_c - t\psi_{\Pi,c,0}}{\rn}=\frac{t}{\rn}\sum_{\la=1}^k\psi_{\Pi,c,\la}\phi_\la,$$
we have using (\ref{normsup}) and since $k \leq l_n$:
\begin{eqnarray*}
\norm{b_{n,k,t}}_{\infty} %&\leq&\frac{t\sqrt{k}}{\rn}\norm{\psi_{\Pi,c,[k]} }_{\ell_2}\\
&\leq&\frac{t\sqrt{k}}{\sqrt{c_0} \rn}\norm{\Pi_{f_0,k}\psi_c - \psi_{\Pi,c,0}}_{f_0}\\
%&\leq&\frac{2t\sqrt{k}}{\rn}\norm{\Pi_{f_0,k}\psi_c}_{f_0,2}\\
%&\leq&\frac{2t\sqrt{k}}{\sqrt{c_0}\rn}\norm{\Pi_{f_0,k}\psi_c}_{f_0,2}\\
%&\leq&\frac{2t\sqrt{k}}{\sqrt{c_0}\rn}\norm{\psi_c}_{f_0,2}\\
&\leq&\frac{2t\sqrt{l_n}}{\sqrt{c_0}\rn}\norm{\psi_c}_{\infty}=O(\ep_n).
\end{eqnarray*}
for $c_0$ a constant. Recall that for $f_\theta \in \mathcal F_k$, $$h_\theta = \rn\left(\sum_{\la\in\Ne^*}(\thl-\thol)\phi_\la - c(\theta)+c(\theta_0)\right)\quad  \mbox{and}
\quad B_{n,k}=\frac{\psi_{\Pi,c,[k]} }{\rn}$$
so, for $\theta'  = \theta - tB_{n,k}$, with $H_n = (h_\theta-t\psi_c)/\rn$ and $\Delta_\psi = \psi_c - \Pi_{f_0,k}\psi_c$
\begin{eqnarray*}
h_{\theta'}  &=& h_\theta - \rn b_{n,k,t} +\rn (c(\theta)-c(\theta - tB_{n,k}))\\
 &=&  h_\theta - t\bar{\psi}_{t,n} + t(\psi_c - \Pi_{f_0,k}\psi_c) - \rn \log \left[\frac{ F_0(e^{H_n+t\Delta_\psi/\rn })}{ F_0(e^{H_n}) }\right] \\
&=&   h_\theta - t\bar{\psi}_{t,n} + t\Delta_\psi - \Delta_n,
 \end{eqnarray*}
with
\begin{eqnarray*}
 \Delta_n&=&\rn \log \left[\frac{ F_0(e^{H_n+t\Delta_\psi/\rn })}{ F_0(e^{H_n}) }\right].
\end{eqnarray*}
Now, (\ref{contc}) implies $|\!|h_\theta|\!|_\infty  / \rn \leq \sqrt{k}\epsilon_n =
o(1)$ and since $F(\Delta_\psi^2)=O(1)$,
$\norm{\Delta_\psi}_\infty=O(\sqrt{l_n})=O(\sqrt{n}\ep_n)$,
\begin{eqnarray*}
F_0(e^{H_n+t\Delta_\psi/\rn })&=&F_0\left(e^{H_n}\left(1+\frac{t\Delta_\psi}{\rn}+\frac{t^2\Delta_\psi^2}{2n}\right)\right)+0\left(F(\Delta_\psi^2)\frac{\norm{\Delta_\psi}_\infty}{n^{3/2}}\right)\\
&=&F_0\left(e^{H_n}\left(1+\frac{t\Delta_\psi}{\rn}+\frac{t^2\Delta_\psi^2}{2n}\right)\right)+0\left(\frac{\ep_n}{n}\right)\\
&=&F_0\left(e^{H_n}\right)+\frac{t}{\rn}F_0(e^{H_n}\Delta_\psi)+\frac{t^2}{2n}F_0(e^{H_n}\Delta_\psi^2)+o\left(\frac{1}{n}\right),
\end{eqnarray*}
Also, for any function $v$ satisfying $F_0(|v|)<\infty$
\begin{eqnarray} \label{appr:Hn}
F_0(e^{H_n}v)&=&F_0\left(v e^{h_\theta/\rn}\right) -\frac{t}{\rn}F_0\left(ve^{h_\theta/\rn}\psi_c\right) +
 O\left(\frac{1}{n}\right).
\end{eqnarray}
Note that in the case $v=1$ since $F_0(e^{h_\theta/\rn})=1$ we can be more precise and obtain
\begin{eqnarray} \label{appr:Hn:1}
F_0(e^{H_n}) &=& 1 - \frac{t}{\rn}F_0\left(e^{h_\theta/\rn}\psi_c\right) + O(1/n) \nonumber \\
 &=& 1 - \frac{t F_0(h_\theta \psi_c )}{ n} +  O\left( \frac{ \epsilon_n^2}{\rn} +
 \frac{1}{n}\right)=1+o\left(\frac{1}{\rn}\right).
 \end{eqnarray}
 Moreover
 \begin{eqnarray} \label{appr:Hn:2}
F_0\left(v e^{h_\theta/\rn}\right) = F_0(v) + o(F_0(|v|)).
 \end{eqnarray}
Therefore using (\ref{appr:Hn})  with $v=\Delta_\psi^2$ leads to
\begin{eqnarray*}
\frac{F_0(e^{H_n+t\Delta_\psi/\rn })}{F_0(e^{H_n})}=1+\frac{t}{\rn}\frac{F_0(e^{H_n}\Delta_\psi)}{F_0(e^{H_n})}+\frac{t^2}{2n}F_0(\Delta_\psi^2)+o\left(\frac{1}{n}\right),
\end{eqnarray*}
and  using (\ref{appr:Hn}) with $v=\Delta_\psi$ together with
(\ref{appr:Hn:1}) and  using (\ref{appr:Hn:2})
\begin{eqnarray*}
 \frac{t}{\rn}F_0(e^{H_n}\Delta_\psi)
&=&\frac{t}{\rn}F_0\left(\Delta_\psi e^{h_\theta/\rn}
\right)-\frac{t^2}{n}F_0\left(\Delta_\psi \psi_c\right)
+o\left(\frac{1}{n}\right).
\end{eqnarray*}
Also
\begin{eqnarray*}
F_0\left(e^{h_\theta/\rn}\Delta_\psi\right) &=& \frac{ 1 }{
\rn}\left[ F_0\left(h_\theta\Delta_\psi\right)  +
\frac{ 1 }{ \rn }F_0\left(h_\theta^2B_{h_\theta,n}\Delta_\psi\right) \right],
\end{eqnarray*}
where $B_{h,n}$ is defined by (\ref{def:Bhn}). 
Since $F_0(e^{h_\theta/\sqrt{n}} \psi_c) = F_0(\psi_c )+o(1)=o(1)$,
we thus obtain using the fact that $F_0(e^{H_n})=1+o\left(\frac{1}{\rn}\right)$ and  $F_0(\psi_c\Delta_\psi)=F_0(\Delta_\psi^2)$
\begin{eqnarray*}
\frac{F_0(e^{H_n+t\Delta_\psi/\rn })}{F_0(e^{H_n})}&=&
1+\frac{t}{\rn}F_0\left(e^{h_\theta/\rn}\Delta_\psi\right)-
\frac{t^2}{2n}F_0\left(\Delta_\psi^2\right)+o\left(\frac{1}{n}\right)
\end{eqnarray*}
and finally,
\begin{eqnarray} \label{Deltan}
\Delta_n&=&\rn\log\left[\frac{F_0(e^{H_n+t\Delta_\psi/\rn })}{F_0(e^{H_n})}\right] \nonumber \\
&=&tF_0\left(e^{h_\theta/\rn}\Delta_\psi\right)-\frac{t^2}{2\sqrt{n}}F_0\left(\Delta_\psi^2\right)
 +o\left(\frac{1}{\sqrt{n}}\right) \nonumber \\
&=& \frac{ t}{\sqrt{n} } \left[ F_0\left(h_\theta \Delta_\psi\right)
+ \frac{F_0\left( h_\theta^2 B_{h_\theta,n} \Delta_\psi \right) }{
\sqrt{n} } - \frac{t}{2}F_0(\Delta_\psi^2) \right] + o(n^{-1/2}).
 \end{eqnarray}
Moreover
 \begin{eqnarray*}
F_0\left( h_\theta^2 B_{h_\theta,n} \Delta_\psi \right) =
\frac{1}{2} F_0\left( h_\theta^2 \Delta_\psi \right) + o\left(
F_0\left( h_\theta^2 |\Delta_\psi| \right)\right)
 \end{eqnarray*}
and by using (\ref{Vbis}),
\begin{eqnarray*}
\frac{ F_0\left( h_\theta^2 |\Delta_\psi| \right) }{ \rn } &\leq&\|\Delta_{\psi}\|_{\infty}\frac{ F_0\left( h_\theta^2  \right) }{ \rn }\\
&\leq&C\|\Delta_{\psi}\|_{\infty}\rn\left(\log n\right)^2\ep_n^2.
\end{eqnarray*}
To bound  $\norm{\Delta_{\psi}}_{\infty}$, we  write
$$\Delta_\psi = \psi_{+k} - \Pi_{f_0,k}(\psi_{+k}),$$
where  $\psi_{+k}$ is a linear function of the $\phi_j$'s for $j\geq k+1$. Then by using (\ref{normsup}),
\begin{eqnarray*}
\norm{\Delta_{\psi}}_{\infty} &\leq& \norm{\psi_{+k}}_\infty +  \norm{\Pi_{f_0,k}\psi_{+k}}_\infty \\
 &\leq&  \norm{\psi_{+k}}_\infty +  C\sqrt{k} \norm{\Pi_{f_0,k}\psi_{+k}}_{f_0} \\
 &\leq& \norm{\psi_{+k}}_\infty +  C\sqrt{k} \norm{\psi_{+k}}_{f_0}\\
 &\leq& \norm{\psi_{+k}}_\infty +  \frac{C\sqrt{\norm{f_0}_\infty} \sqrt{k}}{\sqrt{c_0}}  \norm{\psi_{+k}}_{2}.
\end{eqnarray*}
Under the assumption that
$$\sup_{k \leq l_n}(\norm{\psi_{+k}}_\infty +  \sqrt{k}  |\!|\psi_{+k}|\!|_2) = o\left( \frac{1}{
\rn \epsilon_n^2\left(\log n\right)^{2}} \right),$$
 we obtain that
 $$\Delta_n = \frac{ t}{\sqrt{n} } \left[ F_0\left(h_\theta \Delta_\psi\right)
 - \frac{t}{2}F_0(\Delta_\psi^2) \right] + o(n^{-1/2}).$$
 Note that $\Delta_n = o(1)$. Finally,
\begin{eqnarray*}
R_n(h_{\theta'}) &=& \rn F_0(h_{\theta'}) +\frac{ F_0(h_{\theta'}^2)}{2} \\
 &=& R_n(h_\theta - t\bar{\psi}_{t,n}) - \sqrt{n} \Delta_n - \frac{t^2}{2}F_0( \Delta_\psi^2)
 + tF_0(h_\theta \Delta_\psi)  - \Delta_n F_0(h_\theta) +o(1) \\
 &=&
 R_n(h_\theta - t\bar{\psi}_{t,n})  -\Delta_n F_0(h_\theta) + o(1)
 \end{eqnarray*}
Recall that $h_{\theta'}=h_{\theta}-t\bar{\psi}_{t,n}+t\Delta_\psi-\Delta_n$, $\Delta_n=o(1)$ and $F_0(\Delta_\psi)=0$. Note also that
$$
\bar{\psi}_{t,n}(x) = \psi_c(x) + \frac{\rn}{t} \log\left(F_0(e^{H_n})\right) = \psi_c(x) + o(1)$$
 so that $F_0(\Delta_\psi \bar{\psi}_{t,n}) = F_0(\Delta_\psi^2)+o(1)$ and
\begin{eqnarray*}
-\frac{ F_0(h_{\theta'}^2)}{2}&=& -\frac{ F_0((h_{\theta}-t\bar{\psi}_{t,n})^2)}{2}-\frac{ F_0((t\Delta_\psi-\Delta_n)^2)}{2}-F_0((h_{\theta}-t\bar{\psi}_{t,n})(t\Delta_\psi-\Delta_n))\\
&=&-\frac{ F_0((h_{\theta}-t\bar{\psi}_{t,n})^2)}{2}+\frac{t^2F_0(\Delta_\psi^2)}{2} -tF_0(h_{\theta}\Delta_\psi)+\Delta_nF_0(h_{\theta})+o(1)
\end{eqnarray*}
Furthermore,
$$\G_n(h_{\theta'})=\G_n(h_{\theta}-t\bar{\psi}_{t,n})+t\G_n(\Delta_\psi).$$
We set 
$$\mu_{n,k} =  -F_0(h_\theta \Delta_\psi)  +\G_n(\Delta_\psi)$$ and we 
finally obtain,
\begin{eqnarray*}
\lefteqn{
-\frac{F_0((h_{\theta'})^2)}{2} + \G_n(h_{\theta'}) + R_n(h_{\theta'}) } \\
&&= -\frac{ F_0((h_{\theta}-t\bar{\psi}_{t,n})^2)}{2}+ R_n(h_\theta - t\bar{\psi}_{t,n})+\G_n(h_{\theta}-t\bar{\psi}_{t,n})+t\mu_{n,k} \\
&&\hspace{1cm}+ \frac{t^2F_0(\Delta_\psi^2)}{2}+o(1).
\end{eqnarray*}
  Note that by orthogonality $F_0(h_\theta \Delta_\psi)  =\rn F_0[({\psi_c}-\Pi_{f_0,k}{\psi_c} )\sum_{j\geq k+1}\theta_{0j}\phi_j]$ so that $\mu_{n,k} $ does not depend on $\theta$ and
  setting $T_k\theta = \theta - tB_{n,k}$ for all $\theta$,
 we can write
\begin{eqnarray*}
J_k &:=&\frac{\int_{A_{n}\cap {\mathcal F}_{k}}\exp\left(-\frac{F_0((h_f-t\bar{\psi}_{t,n})^2)}{2}+ G_n(h_f-t\bar{\psi}_{t,n}) +R_n(h_f-t\bar{\psi}_{t,n})\right) d\pi_k(f)}{\int_{A_n\cap {\mathcal F}_{k}}\exp\left(-\frac{F_0(h_f^2)}{2}+ G_n(h_f) +  R_n(h_f)\right)d\pi_k(f)}\\
&=& e^{-\frac{t^2F_0(\Delta_\psi^2)}{2}}
 e^{-t\mu_{n,k}}
  \frac{\int_{\Theta_k\cap A_{n}'} e^{- \frac{F_0\left(h_{T_k\theta}^2\right)}{2}+ \G_n\left(h_{T_k\theta}\right) +  R_n\left(h_{T_k\theta}\right) } d\pi_k(\theta)}
{\int_{\Theta_k\cap A_{n}'} e^{- \frac{F_0(h_\theta^2)}{2}+ \G_n(h_\theta) +  R_n(h_\theta) } d\pi_k(\theta) }(1+o_n(1)),
\end{eqnarray*}
where $A_{n}'=\{\theta:\ f_\theta\in A_n\}.$
Moreover, for $k \leq l_n$, $|\!|B_{n,k}|\!|_2 \leq C/\rn,$
where $C$ depends on $c_0$ and $\norm{\psi_{c}}_\infty$. So, if we set
$$T_k(A_n')   = \{\theta \in \Theta_k\cap A_n': \ \theta+tB_{n,k}\in A_n'\}$$
 for all $\theta \in T_k(A_n')$,
$$|\!|\theta -\theta_0 |\!|_{\ell_2}^2 \leq 2(\log n)^4 \epsilon_n^2 +\frac{2c^2}{n} \leq 2
\epsilon_n^2(\log n)^4(1+o_n(1)) $$ since $n\epsilon_n^2 \rightarrow +\infty$. For all
$\theta \in \Theta_k\cap A_n'$ such that $|\!|\theta -\theta_0 |\!|_2\leq \frac{(\log n)^2 \epsilon_n}{2}$
$$\theta +t B_{n,k} \in A_n' \cap \Theta_k$$ for $n$ large enough and we can write 
$$ A_{n,1}' = \left\{\theta\in A_n':\ |\!|\theta -\theta_0 |\!|_{\ell_2}\leq \frac{(\log n)^2 \epsilon_n}{2} \right\} , \quad A_{n,2}' = \left\{\theta\in A_n':\  |\!|\theta -\theta_0 |\!|_2\leq 3(\log n)^2 \epsilon_n \right\}$$
then
 \begin{eqnarray}\label{Tk:inclusion}
 \Theta_k\cap A_{n,1}' \subset T_k(A_n')
  \subset\Theta_k\cap A_{n,2} '
 \end{eqnarray}
and under assumption (\ref{cond:prior}),
\begin{eqnarray*}
J_k &\leq&  e^{-t^2\frac{F_0(\Delta_\psi^2)}{2}}
 e^{-t \mu_{n,k}}
 \frac{\int_{\Theta_k\cap A_{n,2}'} e^{- \frac{F_0(h_\theta^2)}{2}+ \G_n(h_\theta) +  R_n(h_\theta) } d\pi_k(\theta)}
{\int_{\Theta_k\cap A_n' } e^{- \frac{F_0(h_\theta^2)}{2}+ \G_n(h_\theta) +  R_n(h_\theta) } d\pi_k(\theta) }(1+ o_n(1)), \\
 J_k &\geq & e^{-t^2\frac{F_0(\Delta_\psi^2)}{2}}
 e^{-t\mu_{n,k}}
  \frac{\int_{\Theta_k\cap A_{n,1}' } e^{- \frac{F_0(h_\theta^2)}{2}+ \G_n(h_\theta) +  R_n(h_\theta) } d\pi_k(\theta)}
{\int_{\Theta_k\cap A_n' } e^{- \frac{F_0(h_\theta^2)}{2}+ \G_n(h_\theta) +  R_n(h_\theta) } d\pi_k(\theta) }(1+ o_n(1)).
\end{eqnarray*}
Therefore,
\begin{eqnarray*}
\zeta_n(t) &:=& \esp[\exp(t\rn(\psi(f) - \psi(\Pb_n))) \1_{A_n}(f) | X^n ] \\
 &=&e^{\frac{t^2F_0(\psi_c^2)}{2}}
\left[ \sum_{k=1}^{l_n} p(k|X^n) J_k \right](1+ o_n(1)) \\
 &\leq&
\left[ \sum_{k=1}^{l_n} p(k|X^n)
\1_{\Theta_k \cap A_n' \neq \emptyset}e^{-t \mu_{n,k}}e^{t^2\frac{F_0({\psi_c}^2)- F_0(\Delta_\psi^2) }{2}}
\right](1+ o_n(1))
\end{eqnarray*}
and
\begin{eqnarray*}
 \zeta_n(t) &\geq& e^{t^2\frac{F_0({\psi_c}^2)}{2}}
 \sum_{k=1}^{l_n} p(k|X^n) e^{-t \mu_{n,k}}e^{-t^2\frac{F_0(\Delta_\psi^2) }{2}}\pi\left[  A_{n,1}'|X^n,k\right].
 \end{eqnarray*}
Besides under the above conditions on the prior, with probability converging to 1,
$$\pi\left[  (A_{n,1}')^c|X^n\right] \leq  e^{-nc\epsilon_n^2},$$
for some positive constant $c>0$. Then uniformly over $k $ such that $\Theta_k \cap A_{n,1}' \neq \emptyset$
$$\pi\left[  (A_{n,1}')^c|X^n,k\right]e^{-t\mu_{n,k} } = o(1)$$
and
\begin{eqnarray*}
\zeta_n(t) &\geq&
e^{t^2\frac{F_0({\psi_c}^2)}{2}}\sum_{k=1}^{l_n} p(k|X^n)
\1_{\Theta_k \cap A_n \neq \emptyset}e^{-t \mu_{n,k}} e^{-t^2\frac{F_0(\Delta_\psi^2) }{2}}(1+
o_n(1)).
 \end{eqnarray*}
This proves that the posterior distribution of $\rn(\Psi(f) -
\Psi(P_n))$ is asymptotically equal to a mixture of Gaussian
distributions with variances $V_{0k} =F_0({\psi_c}^2) - F_0(\Delta_\psi^2)$, means $-\mu_{n,k}$ and weights $p(k|X^n)$.

Now if $|\!|\Delta_\psi|\!| = o(1)$ ($k \rightarrow +\infty$) $\G_n(\Delta_{\psi}) = o_P(1)$ and with probability converging to 1,
 \begin{eqnarray*}
 |\mu_{n,k}| &\leq& |\!| f_0|\!|_{\infty}\rn \left(\sum_{j=k+1}^{+\infty} \psi_{c,j}^2 \right)^{1/2}
 \left( \sum_{j=k+1}^{+\infty} \theta_{0j}^2 \right)^{1/2}+o_n(1).
\end{eqnarray*}
 Thus if $k=k_n^*$,
\begin{eqnarray*}
 |\mu_{n,k}| &=& o\left(\rn (k_n^*)^{-\gamma-1/2}\right)+o_n(1)=o_n(1)
\end{eqnarray*}
 and Equality
(\ref{BVM}) is proved.
\\\\
\noindent{\bf Ackowledgment}:  This work has been partially supported by the ANR-SP Bayes grant
%%%%%%%%%%%

\end{document}